\documentclass[hidelinks,onefignum,onetabnum]{siamart251216}

\ifpdf
\hypersetup{
pdftitle={A black-box, multilevel algebraic preconditioning framework for conforming finite elements},
pdfauthor={O. A. Krzysik, B. S. Southworth, and G. A. Wimmer}
}
\fi

\usepackage{amsmath,amssymb,amsfonts}
\usepackage{mathtools}
\mathtoolsset{centercolon}

\usepackage{bm}
\usepackage{xcolor}
\usepackage{graphicx}
\usepackage{epstopdf}
\usepackage{algorithmic}
\usepackage{tikz}
\usepackage{subcaption}
\usepackage{array}
\usepackage{comment}
\usepackage[normalem]{ulem}

\usetikzlibrary{calc,positioning,decorations.pathreplacing,arrows.meta,patterns,fit,matrix}

\ifpdf
\DeclareGraphicsExtensions{.pdf,.png,.jpg,.eps}
\else
\DeclareGraphicsExtensions{.eps}
\fi

\headers{Algebraic multilevel solvers for conforming finite elements}
{O. A. Krzysik, B. S. Southworth, and G. A. Wimmer}

\title{A black-box, multilevel algebraic preconditioning framework for conforming finite elements%
\thanks{\funding{This work was supported by the Laboratory Directed Research and Development program of Los Alamos National Laboratory, under project number 20240261ER. Los Alamos National Laboratory report number LA-UR-26-25292.}}}

\author{
Oliver A. Krzysik\thanks{Theoretical Division, Los Alamos National Laboratory
(\email{okrzysik@lanl.gov}, \url{https://orcid.org/0000-0001-7880-6512}.}
\and
Ben S. Southworth\thanks{Theoretical Division, Los Alamos National Laboratory
(\email{southworth@lanl.gov}, \url{https://orcid.org/0000-0002-0283-4928}.}
\and
Golo A. Wimmer\thanks{Theoretical Division, Los Alamos National Laboratory
(\email{gwimmer@lanl.gov}, \url{https://orcid.org/0000-0002-7871-1748}.}
}

\newsiamthm{remark}{Remark} %
\newsiamremark{hypothesis}{Hypothesis}
\newsiamremark{assumption}{Assumption}
\newsiamremark{fact}{Fact}
\newsiamthm{claim}{Claim}

\crefname{hypothesis}{Hypothesis}{Hypotheses}
\Crefname{hypothesis}{Hypothesis}{Hypotheses}
\crefname{assumption}{Assumption}{Assumptions}
\Crefname{assumption}{Assumption}{Assumptions}
\crefname{fact}{Fact}{Facts}
\Crefname{fact}{Fact}{Facts}
\crefname{claim}{Claim}{Claims}
\Crefname{claim}{Claim}{Claims}

\crefname{section}{Section}{Sections}
\Crefname{section}{Section}{Sections}
\crefname{subsection}{Section}{Sections}
\Crefname{subsection}{Section}{Sections}

\newcolumntype{P}[1]{>{\raggedright\arraybackslash}p{#1}}

\allowdisplaybreaks

\newcommand{\hcurl}{H(\operatorname{curl})}
\newcommand{\hdiv}{H(\operatorname{div})}

\newcommand{\wh}[1]{\widehat{#1}}
\newcommand{\wt}[1]{\widetilde{#1}}

\DeclareFontFamily{U}{mathx}{}
\DeclareFontShape{U}{mathx}{m}{n}{<-> mathx10}{}
\DeclareSymbolFont{mathx}{U}{mathx}{m}{n}
\DeclareMathAccent{\wc}{0}{mathx}{"71}
\DeclareMathAccent{\wbar}{0}{mathx}{"73}

\DeclareMathOperator{\rank}{rank}

\DeclareMathOperator{\grad}{grad}

\let\div\relax
\DeclareMathOperator{\div}{div}

\begin{document}
\maketitle

\begin{abstract}
Recently we introduced the least-squares algebraic-multigrid domain-decomposition (LS-AMG-DD) method as a multilevel, algebraic preconditioner for sparse symmetric positive definite (SPD) matrices that admit a Gram representation \(A=G^{\top}G\) \cite{southworth2026lsamgdd}. 
The factor \(G\) induces a local symmetric positive semidefinite (SPSD) splitting of \(A\) used to define local spectral problems from which an interpolation $P$ is built, and a coarse-level Gram operator induced under Galerkin coarsening, \(A_c=G_c^\top G_c\), for \(G_c:=GP\).
This paper clarifies when this Gram structure arises, showing that, on a prescribed degree-of-freedom cover \({\cal C}\), a \({\cal C}\)-local Gram representation of $A$ exists if and only if \(A\) admits a \({\cal C}\)-local SPSD splitting. 
We then connect this viewpoint to conforming finite-element discretizations, where bilinear forms are naturally assembled from elementwise SPSD energies and therefore admit element-local Gram representations after choosing local factors (e.g., via algebraic factorizations of element blocks).
Taken together, these observations provide an essentially black-box route for applying LS-AMG-DD to conforming finite-element problems. Numerical tests illustrate the robustness of the method on several problems for which classical AMG methods require more than $10^5$ iterations to converge, including high-order discretizations of grad--div in \(\hdiv\), anisotropic hyperdiffusion in $H^2$, and linear elasticity in vector \(H^1\). Moreover, in some comparisons with existing AMG methods, LS-AMG-DD produces errors that are 2--5 orders of magnitude smaller, even when all methods are stopped at the same relative residual tolerance.
\end{abstract}



\section{Introduction}

Robust algebraic solvers for large sparse symmetric positive definite (SPD) systems remain a central need in scientific computing. In real application workflows, often an algebraic solver is necessary, as a user or application code may be unable or unwilling to provide more detailed information, such as a full geometric mesh hierarchy or a specific physical problem formulation a priori. Moreover, robustness is critical to ensure that a full multiphysics simulation does not fail due to sensitivity or lack of robustness of a single inner linear solver. For scalar elliptic problems posed in $H^1$, algebraic multigrid (AMG) is often the method of choice in high-performance codes, and classical AMG and smoothed aggregation can deliver near-optimal complexity on isotropic, mildly anisotropic, and heterogeneous problems \cite{xu2017amg,vanek1996sa}. However, even restricted to SPD systems that arise in the context of numerical partial differential equations (PDEs), AMG can quickly fail when confronted with complexities such as coupled systems of PDEs, finite-element spaces aside from $H^1$, strong anisotropy and directional behavior, or high-order discretizations. 

These challenges all relate to standard assumptions underlying many AMG methods around near-nullspace modes, i.e., modes $\bm e$ for which $A \bm e \approx \bm 0$. Under typical pointwise smoothing, these are the slowest modes to converge, and also the modes that the coarse-grid correction is required to capture for a successful two-level method. First, it is typically assumed that near-nullspace modes that are slow to converge under the smoother vary slowly in the direction of strong connections of the matrix. Equivalently, the modes that need to be represented in interpolation can be identified with locally smooth vectors across strong matrix connections. Second, it is assumed that the large near nullspace that the range of interpolation must contain can be represented by a direct sum over local coarse basis functions. However, these assumptions can break down. For example, in highly anisotropic diffusion there are near-nullspace modes that are oscillatory in the direction of weak diffusion \cite{wimmer2024fast}, which do not necessarily correspond to slow variation in the direction of strong matrix connections, and also cannot be attenuated by local smoothing \cite{krzysik2026overlapping}. Similarly, classical grad--div and Maxwell-type problems posed in $\hdiv$ and $\hcurl$ have large near nullspaces that cannot be represented by a direct sum of local basis functions. 

Many specialized AMG methods, and more broadly multigrid-like methods, have been developed to handle these various challenges. Some work has developed black-box AMG methods for systems of PDEs, e.g. \cite{clees2005amg,magri2025multigrid}, but coupled systems are still typically handled with block preconditioning, with multigrid used to approximate inverses within the block preconditioner. For extreme anisotropy in fusion, specialized techniques have been developed, e.g. \cite{green2024efficient,wimmer2024fast}, and certain other problems like elasticity have known physical near-nullspace modes that interpolation can be constrained to fit exactly, such that they can be represented accurately on coarse levels, e.g. \cite{baker2010elasticity}. In the geometric multigrid context of $\hdiv$ and $\hcurl$, block smoothers are used that exploit exact-sequence and/or mesh-entity structure \cite{arnold2000hdivhcurlmg,hiptmair1997hdivmg,hiptmair1998maxwellmg}. In the algebraic setting, the auxiliary-space framework of Hiptmair--Xu \cite{hiptmair2007nodal} has led to the development of effective $\hdiv$ and $\hcurl$ solvers, in the form of ADS and AMS, respectively \cite{kolev2009ams,kolev2012ads}.
A related AMG construction based on the preservation of discrete nullspaces in the $\hcurl$ context was developed by Reitzinger and Sch{\"o}berl \cite{reitzinger2002algebraic}, and has since been improved through the use of smoothed aggregation \cite{bochev2003improved,hu2006toward,tuminaro2025structure}.
Other approaches include algebraic hybridization \cite{dobrev2019algebraic}, and ongoing work on more direct AMG-style constructions \cite{shen2026nodal}. Variants in $p$-multigrid have been studied since the late 1980s \cite{ronquist1987spectral}, and have seen renewed interest lately in high-order finite elements, e.g. \cite{pazner2020efficient}. The aforementioned works provide powerful tools that are often effective, but, in general, they rely on nontrivial discretization-aware data, are somewhat problem/difficulty specific, and their implementation burden is higher.

Another potential route towards robust algebraic solvers is via spectral coarse spaces, which are built from local low-energy modes that arise as solutions of carefully designed local generalized eigenvalue problems. 
This principle appears in overlapping Schwarz coarse spaces (including GenEO-type constructions) \cite{spillane2014geneo}, algebraic domain decomposition methods based on local SPSD splittings \cite{aldaas2019coarsespaces,aldaas2021multilevel,aldaas2022normaleq}, element-based AMG \cite{brezina2001amge,chartier2003ramge}, and generalized multiscale finite-element methods \cite{efendiev2013gmsfem}. 
Although these methods differ in viewpoint and implementation, they share a common underlying philosophy \cite{southworth2026lsamgdd}. 
A limiting aspect of these methods is that most are naturally two-level because the structure exploited to coarsen the finest level does not propagate to the second level.

Our recently developed least-squares algebraic-multigrid domain-decomposition (LS-AMG-DD) solver \cite{southworth2026lsamgdd} overcomes the two-level limitation of the aforementioned methods in the setting of sparse Gram operators $A=G^{\top}G$, as arise in least-squares problems. 
The Gram factor $G$ provides a mechanism for building an overlap-local SPSD splitting of $A$, and hence local spectral problems, and, importantly, is naturally multilevel, since the Gram structure propagates under Galerkin coarsening, $A_c= G_c^\top G_c$, with $G_c := GP$ for interpolation $P$. 
The resulting method combines aggregation-based coarsening, overlapping Schwarz smoothing, and local spectral enrichment in a fully algebraic multilevel construction \cite{southworth2026lsamgdd}. At a high level, the objective of LS-AMG-DD is to provide the robustness of spectral DD methods, as well as the efficiency, scalability, and low operator complexities associated with AMG (where feasible to do so).
In our companion paper \cite{krzysik2026theory}, we develop a fully algebraic two-level convergence theory for LS-AMG-DD, showing that the coarse space satisfies a multigrid weak approximation property in a block smoother norm, and that the convergence rate of the method is essentially controlled by the (user-specified) local spectral cutoff threshold, and the spectral radius of the preconditioned smoothing iteration.

This paper presents a novel framework for solving conforming finite-element problems by revealing that the Gram structure required by LS-AMG-DD \cite{southworth2026lsamgdd} arises naturally in standard conforming finite-element assembly. The key observation is that standard conforming finite-element discretizations are built by assembling local semicoercive quadratic energies on elements. This immediately provides element-local Gram representations, which naturally induce a global Gram structure.
From a ``black-box'' perspective, our framework simply requires the matrix $G$ from the  representation $A=G^{\top}G$, which we show how to construct in this paper using basic finite-element assembly information. 
In particular, this framework does \emph{not} require discretization- or problem-specific structures such as element-agglomeration data used by AMGe \cite{brezina2001amge,chartier2003ramge}, or auxiliary exact-sequence information as used by AMS/ADS \cite{kolev2009ams,kolev2012ads}, or nullspace vectors as commonly required by AMG methods for PDE systems \cite{vanvek1996algebraic,baker2010elasticity}, or general information on the PDE and/or finite-element spaces. Although we use the LS-AMG-DD solver throughout this paper, many of the ideas presented are not specific to this solver, and could be used to motivate new solvers in future work. More broadly, we show that conforming finite-element assembly can be viewed as a source of local Gram structure, which is trivially preserved under algebraic coarsening, and we also show that algebraically defined overlapping subdomains induced from this Gram structure have a direct geometric interpretation as closures of finite-element mesh entities.

The remainder of this paper is organized as follows.
\Cref{sec:background} presents relevant background details on the LS-AMG-DD algorithm.
\Cref{sec:local_gram_existence} proves an abstract equivalence between local Gram representations and exact local SPSD splittings. Background on conforming finite-element problems is provided in \Cref{sec:conforming_discr}, and elementwise Gram constructions for conforming discretizations is presented in
\Cref{sec:elementwise_gram_formulation}. In \Cref{sec:numerics}, we demonstrate LS-AMG-DD to be a robust black-box AMG solver for three challenging conforming finite-element problems that classical AMG methods fail to converge on within $10^5$ iterations: grad-div posed in $\hdiv$, anisotropic biharmonic in $H^2$, and vector elasticity, with all problems considered across multiple finite-element orders and physical parameters (reaction coefficient, anisotropy ratio, and Poisson ratio, respectively). For each problem, LS-AMG-DD not only reaches relative residual tolerances substantially faster than existing AMG methods, but at a fixed relative residual tolerance also provides as much as five orders of magnitude smaller error, showing the importance of a robust preconditioner that can uniformly attenuate all error modes. 
Conclusions are discussed in \cref{sec:conclusions}.

\section{Background on LS-AMG-DD}
\label{sec:background}

LS-AMG-DD is a fully algebraic multilevel preconditioner for sparse SPD
operators admitting a Gram representation \(A=G^{\top}G\)
\cite{southworth2026lsamgdd}.  We describe the two-level construction first;
the same construction is applied recursively on coarse levels because the
Gram structure is preserved by Galerkin coarsening, as described below.
Let \(A\in\mathbb R^{n\times n}\) be SPD and let
\(G\in\mathbb R^{m\times n}\), with \(m\ge n\), satisfy \(A=G^{\top}G\).
Here \(n\) is the number of degrees of freedom (DOFs), while \(m\) is the
number of equations represented by the rows of \(G\).  Since \(A\) is SPD,
\(G\) has full column rank. We consider vectors as column vectors and let $G$ take the form
\begin{equation}
\label{eq:G_row_convention}
    G
    =
    \begin{bmatrix}
        \bm g_1^{\top} \\
        \vdots \\
        \bm g_m^{\top}
    \end{bmatrix},
    \qquad
    G^{\top}
    =
    \begin{bmatrix}
        \bm g_1 & \cdots & \bm g_m
    \end{bmatrix},
    \qquad
    \bm g_j\in\mathbb R^n .
\end{equation}
Thus \(A=G^{\top}G=\sum_{j=1}^m\bm g_j\bm g_j^{\top}\), so each row of
\(G\) contributes one SPSD rank-one row energy to the global operator.

The two-level LS-AMG-DD cycle has the standard multiplicative error
propagator
\[
    E_{\rm TG}
    =
    \bigl(I-M^{-\top}A\bigr)
    \bigl(I-PA_c^{-1}P^{\top}A\bigr)
    \bigl(I-M^{-1}A\bigr),
    \quad
    A_c:=P^{\top}AP = (G P)^\top (G P).
\]
Here \(P\in\mathbb R^{n\times n_c}\) is the interpolation operator described
below, and \(M^{-1}\) is a block Schwarz smoother.
Standard sharp two-level theory \cite{falgout2005twogrid} gives 
\(
\|E_{\rm TG}\|_A = 1 - 1/K_{\rm TG},
\)
where \(K_{\rm TG}\) is the weak-approximation constant measured in the norm induced by the symmetrized smoother \(\wt M\). Further, \(K_{\rm TG}\) bounds the condition number of the associated preconditioned operator $B_{\rm TG}^{-1} A$, defined such that $E_{\rm TG} = I - B_{\rm TG}^{-1} A$.
In \cite{krzysik2026theory} we show that, for any convergent smoother \(M^{-1}\),
\(
K_{\rm TG}
\le
\lambda_{\max}(M_{\omega{\rm J}}^{-1}\wt M)\tau_{\max}
\),
with \(\tau_{\max}\) the largest realized local eigenvalue cutoff (described below).
Here, \(\lambda_{\max}(M_{\omega{\rm J}}^{-1}\wt M)\) represents the cost of transferring the approximation property from an aggregate-wise block-Jacobi norm, for aggregates $\{ \omega_i \}$ defined below, to the symmetrized smoother norm.
The original LS-AMG-DD in \cite{southworth2026lsamgdd} used a restricted additive Schwarz (RAS) smoother, but this choice is not intrinsic to the coarse-space construction of the method.
In particular, one can use any block-style Schwarz smoother such as overlapping multiplicative Schwarz, or block Jacobi; moreover, RAS smoothing is not always well defined, and can produce a divergent two-level method \cite{krzysik2026theory}. 

The DOFs \(\{1,\ldots,n\}\) are partitioned into disjoint aggregates
\begin{equation}
\label{eq:agg_def}
    \{\omega_i\}_{i=1}^{n_{\rm agg}},
    \qquad
    \bigcup_{i=1}^{n_{\rm agg}}\omega_i=\{1,\ldots,n\},
    \qquad
    \omega_i\cap\omega_k=\emptyset
    \quad\text{for }i\ne k .
\end{equation}
In practice we typically do this partitioning by applying a standard AMG aggregation routine to $A$.
For a DOF set \(X\subseteq\{1,\ldots,n\}\), we write
\(R_X:=I_n(X,:)\) for the Boolean restriction to the coordinates indexed by
\(X\).  For a row \(\bm g_j^{\top}\) of $G \in \mathbb{R}^{m \times n}$ in \eqref{eq:G_row_convention}, define its DOF support by
\(\operatorname{supp}(\bm g_j):=\{k:(\bm g_j)_k\ne0\}\).  Given any
\(\omega\subseteq\{1,\ldots,n\}\), define its row-touching set by
\begin{equation}
\label{eq:calR_def}
    \mathfrak R_{\omega}
    :=
    \left\{
        j\in\{1,\ldots,m\}:
        \operatorname{supp}(\bm g_j)\cap\omega\neq\emptyset
    \right\}.
\end{equation}
For each \(\omega_i\), the associated overlap is taken to be the corresponding row closure,\footnote{This closure is closely related to adding the distance-one neighbors of \(\omega_i\) in the graph of \(A\). The distinction between the two is usually minor, but can be significant; see \cref{sec:agg_closure}.}
\begin{equation}
\label{eq:overlap_def}
    \Omega_i
    :=
    \bigcup_{j\in\mathfrak R_{\omega_i}}
    \operatorname{supp}(\bm g_j),
    \qquad
    \Gamma_i:=\Omega_i\setminus\omega_i .
\end{equation}
With DOFs ordered first by aggregate then interface, define
\(
D_{\Omega_i}:=
\begin{bsmallmatrix}
I_{|\omega_i|} & 0\\
0 & 0_{|\Gamma_i|}
\end{bsmallmatrix}.
\)

The factor $G$ provides a natural local SPSD splitting over the overlaps
\(\{\Omega_i\}_{i=1}^{n_{\rm agg}}\). To this end, we define the aggregate-touching multiplicity of the $j$th row of $G$ by
\begin{equation}
\label{eq:row_multiplicity}
    \operatorname{mult}_{\omega}(j)
    :=
    \#\left\{
        i\in\{1,\ldots,n_{\rm agg}\}:
        j\in\mathfrak R_{\omega_i}
    \right\},
\end{equation}
noting that any one row of \(G\) may touch more than one aggregate.
After removing identically zero rows of \(G\), every row touches at least one
aggregate, so \(\operatorname{mult}_{\omega}(j)\ge1\).  For each aggregate,
define
\begin{align} \label{eq:Wmatrix-def}
    W_{\mathfrak R_{\omega_i}}
    :=
    \operatorname{diag}_{j\in\mathfrak R_{\omega_i}}
    \bigl(\operatorname{mult}_{\omega}(j)^{-1}\bigr),
    \qquad
    H_{\Omega_i}
    :=
    W_{\mathfrak R_{\omega_i}}^{1/2}
    G(\mathfrak R_{\omega_i},\Omega_i).
\end{align}
The multiplicity weights distribute each row energy uniformly over the
aggregates touched by that row.  Consequently, we have the following exact SPSD splitting
\begin{equation}
\label{eq:SPSD_splitting}
    A
    =
    \sum_{i=1}^{n_{\rm agg}}
    R_{\Omega_i}^{\top}
    \widetilde A_{\Omega_i}
    R_{\Omega_i},
    \qquad
    \widetilde A_{\Omega_i}
    :=
    H_{\Omega_i}^{\top}H_{\Omega_i}
    \succeq0 .
\end{equation}
The local matrices \(\widetilde A_{\Omega_i}\) are therefore assembled from
weighted rows of \(G\), rather than by simply restricting the global matrix
\(A\) to \(\Omega_i\).

The interpolation \(P\) is block diagonal over the aggregates
\(\{\omega_i\}_{i=1}^{n_{\rm agg}}\), and is constructed as follows.  
For each overlap, let
\(A_{\Omega_i}:=R_{\Omega_i}AR_{\Omega_i}^{\top}\) be the principal submatrix of $A$ on $\Omega_i$.  
Consider the local generalized
eigenvalue problem
\begin{equation}
\label{eq:gep_local}
    D_{\Omega_i}A_{\Omega_i}D_{\Omega_i}\bm u_i
    =
    \lambda\,\widetilde A_{\Omega_i}\bm u_i .
\end{equation}
For finite nonzero eigenvalues, eliminating the interface variables gives the following 
reduced problem on the aggregate DOFs,
\begin{equation}
\label{eq:gep_reduced}
    A_{\omega_i,\omega_i}\bm u_{\omega_i}
    =
    \lambda\,\widetilde S_i\bm u_{\omega_i},
    \qquad
    \widetilde S_i
    :=
    \widetilde A_{\omega_i,\omega_i}
    -
    \widetilde A_{\omega_i,\Gamma_i}
    \bigl(\widetilde A_{\Gamma_i,\Gamma_i}\bigr)^{\dagger}
    \widetilde A_{\Gamma_i,\omega_i}.
\end{equation}
Here \(\widetilde S_i\) is the Schur complement of
\(\widetilde A_{\Omega_i}\) onto \(\omega_i\), and \(\dagger\) denotes the
Moore--Penrose pseudoinverse.  The local coarse space is built from local generalized eigenmodes with eigenvalue
above a prescribed spectral threshold $ \tau_{\rm cut}$, together with the \(+\infty\) modes
corresponding to \(\ker(\widetilde S_i)\).  
In terms of the convergence theory from \cite{krzysik2026theory} discussed above, one has \(\tau_{\max} \leq \tau_{\rm cut}\), where \(\tau_{\max}\) is the realized largest discarded eigenvalue.  
If \(Z_i\) denotes the retained
basis on \(\omega_i\), then
\[
    P
    =
    \begin{bmatrix}
        R_{\omega_1}^{\top}Z_1
        \;\cdots\;
        R_{\omega_{n_{\rm agg}}}^{\top}Z_{n_{\rm agg}}
    \end{bmatrix}.
\]
Thus the columns of \(P\) have nonoverlapping aggregate support, while the
eigenproblems used to define those columns are posed on the overlapping
subdomains \(\Omega_i\). In this paper we propose an adaptive local eigenvalue threshold 
\begin{equation}\label{eq:tau-cut}
    \tau_{\rm cut} \propto \max_j {\rm mult}_{\omega}(j),
\end{equation}
with an order-one proportionality constant.
This choice is motivated by trying to balance convergence rates with operator complexity in view of the solution structure of the local eigenvalue problems \eqref{eq:gep_reduced}, although for brevity we do not discuss in more detail here.

Finally, the multilevel construction follows because Galerkin coarsening
preserves the Gram form:
\(
    A_c
    =
    P^{\top}AP
    =
    P^{\top}G^{\top}GP
    =
    (GP)^{\top}(GP).
\)
Thus one sets \(G_c:=GP\), so that \(A_c=G_c^{\top}G_c\).  The same
aggregation, row-based SPSD splitting, local spectral construction, and block
Schwarz smoothing can then be repeated on the coarse level.  Recursing in
this way gives the multilevel LS-AMG-DD hierarchy, where a direct solve is used on the
coarsest level.

\section{The existence of sparse Gram representations}
\label{sec:local_gram_existence}

In the LS-AMG-DD paper \cite{southworth2026lsamgdd}, the required Gram structure \(A=G^{\top}G\) appears through several ad hoc constructions.
For example, it is noted that scalar elliptic operators can be factored along the lines of \(- \div \grad = ( \grad )^\top ( \grad ) \), or that Schur complements of certain \(2\times 2\) mixed finite-element formulations can be well approximated by stacked Gram representations. 
More generally, given only a sparse SPD matrix \(A\), when should one expect a \emph{local} Gram representation to exist in the first place?
In this section we address this question and show that the above examples are instances of a general equivalence between local Gram structure and local SPSD energy splittings.

Note that every SPD matrix admits a global Gram representation, via e.g., the square root \(A=A^{1/2}A^{1/2} = G^\top G\), or Cholesky \(A=LL^{\top}=G^{\top}G\). 
What LS-AMG-DD requires, at least to be practical, is \emph{locality}, i.e. a representation whose \emph{rows} are supported on prescribed small subsets of DOFs. This motivates the following definition of a DOF cover.
\begin{definition}[DOF cover] \label{def:cover}
Given $A \in \mathbb{R}^{n \times n}$, we define a \emph{cover} ${\cal C}$ of its DOF set $\{1, \ldots, n\}$ such that
\[
{\cal C}=\{U_\alpha\}_{\alpha\in I},
\qquad
\bigcup_{\alpha\in I}U_\alpha=\{1,\dots,n\},
\]
with each cover member \(U_\alpha\subseteq\{1,\dots,n\}\) a DOF subset. 
For each \(U_\alpha\in{\cal C}\), let \({R_\alpha : \mathbb{R}^n \to \mathbb{R}^{|U_{\alpha}|}}\) denote Boolean restriction from $\{1, \ldots, n\}$ to \(U_\alpha\).
If ${G\in\mathbb{R}^{m\times n}}$ is as in \eqref{eq:G_row_convention}, then we say \(G\) is \emph{\({\cal C}\)-local} if
\[
\forall j\in\{1,\dots,m\}\ \ \exists\,\alpha(j)\in I \ \ \text{such that}\ \ \operatorname{supp}(\bm{g}_j)\subseteq U_{\alpha(j)}.
\]
That is, \(G\) is \emph{\({\cal C}\)-local} if each row \(\bm{g}_j^\top\) is supported entirely inside \emph{at least one} member of the cover.
\end{definition}

Note that at this stage the DOF sets \(U_\alpha\), $\alpha \in I$, are purely algebraic; for example, they could be induced by a discretization (e.g.\ the DOFs supported on a mesh element) or by a solver construction (e.g.\ overlapping Schwarz patches).
Note also that because cover elements can overlap, a row of $G$ may intersect multiple cover members, but the key point is that it is fully supported in at least one member.  
\subsection{Equivalence between local Gram representations and local SPSD splittings}

The key takeaway of this section is that local Gram structure is equivalent to local SPSD quadratic-energy structure, as in the following theorem.

\begin{theorem}[Local SPSD splitting \(\Longleftrightarrow\) local Gram]
\label{thm:local_split_iff_local_gram}
Let \(A\in\mathbb R^{n\times n}\) be SPD, and let \({\cal C}=\{U_\alpha\}_{\alpha\in I}\) be a cover of its DOF set as in \cref{def:cover}. 
Then the following are equivalent:
\begin{enumerate}
\item there exist matrices \(\widetilde A_\alpha\in\mathbb R^{|U_\alpha|\times |U_\alpha|}\) with \(\widetilde A_\alpha\succeq0\) such that
\begin{equation}
\label{eq:cover_local_spsd_split}
A=\sum_{\alpha\in I} R_\alpha^{\top}\widetilde A_\alpha R_\alpha;
\end{equation}
\item there exists a \({\cal C}\)-local Gram representation \(A=G^{\top}G\).
\end{enumerate}
\end{theorem}

\begin{proof}
\textbf{Case 1:} ${\cal C}$-local SPSD splitting $\Longrightarrow {\cal C}$-local Gram.
Assume \eqref{eq:cover_local_spsd_split}. For each \(\alpha\), factor \(\widetilde A_\alpha=B_\alpha^{\top}B_\alpha\), and define \(G\) by stacking the matrices
\(B_\alpha R_\alpha\) vertically. Then
\[
G^{\top}G
=
\sum_{\alpha\in I} R_\alpha^{\top}B_\alpha^{\top}B_\alpha R_\alpha
=
\sum_{\alpha\in I} R_\alpha^{\top}\widetilde A_\alpha R_\alpha
=
A,
\]
and every row of \(B_\alpha R_\alpha\) is supported in \(U_\alpha\), so \(G\) is \({\cal C}\)-local.

\textbf{Case 2:} ${\cal C}$-local Gram $\Longrightarrow {\cal C}$-local SPSD splitting.
Assume \(A=G^{\top}G\) as in \eqref{eq:G_row_convention}, with \(G\in\mathbb R^{m\times n}\) \({\cal C}\)-local. 
Our proof proceeds by regrouping the row outer products \( \bm{g}_j \bm{g}_j^{\top}\) into patch-local contributions.

\emph{Step 1 (assign each row outer product to a patch that contains it).}
Recall from \cref{def:cover} that \({\cal C}\)-locality of $G \in \mathbb{R}^{m \times n}$ means that for each row index \(j \in \{1, \ldots, m\}\) there exists at least one \(\alpha\in I\) with \(\operatorname{supp}(\bm{g}_j)\subseteq U_\alpha\).
Based on this, we now assign each row to one of the cover members that supports it.
Specifically, let $J_\alpha \subseteq \{ 1, \ldots, m\}$, $\alpha \in I$, be a subset of rows such that \(\operatorname{supp}(\bm{g}_j)\subseteq U_\alpha\) for every \(j\in J_\alpha\), and, suppose, moreover that every row is assigned to exactly one subset (i.e., the sets \(\{ J_{\alpha} \}_{\alpha \in I}\) cover $\{1, \ldots, m\}$ and are pairwise disjoint).

\emph{Step 2 (restrict to local coordinates and reassemble).}
Fix \(\alpha\), and suppose that $J_{\alpha} = \{j_1, \ldots, j_{|J_{\alpha}|}\}$.
Then let
\(G_\alpha
:=
G(J_\alpha,:)
\)
be the submatrix of $G$ consisting of rows in \(J_\alpha\), 
and define
\(
\wh{G}_\alpha
:=
G_\alpha R_\alpha^\top
\).
Now, since every row in \(J_\alpha\) is fully supported in \(U_\alpha\), we lose no information from Boolean restriction of each row from $\{1, \ldots, n \}$ onto \(U_\alpha\) and then embedding it back again.
That is, restricting \(\bm g_j\) to \(U_\alpha\)
and embedding back recovers \(\bm g_j\), i.e.\ \(\bm g_j=R_\alpha^\top\wh{\bm g}_{j\alpha}\) with \(\wh{\bm g}_{j\alpha}=R_\alpha\bm g_j\).
Hence
\(
G_\alpha^{\top} G_\alpha 
=
R_\alpha^{\top}\wh{G}_\alpha^{\top}\wh{G}_\alpha R_\alpha
\)

Summing over \(\alpha\), using that every row index belongs to a single index subset \(J_{\alpha}\), gives
\[
A
=G^{\top}G
=\sum_{\alpha\in I} G_\alpha^{\top}G_\alpha
=\sum_{\alpha\in I} 
R_\alpha^{\top}\wh{G}_\alpha^{\top}
\wh{G}_\alpha R_\alpha
\]
which is \eqref{eq:cover_local_spsd_split} with \(\widetilde A_\alpha:= \wh{G}_\alpha^{\top}
\wh{G}_\alpha \succeq 0\).
\end{proof}

Thus \cref{thm:local_split_iff_local_gram} gives a complete characterization of the problem on a \emph{fixed prescribed cover} \({\cal C}\): an SPD matrix admits a \({\cal C}\)-local Gram representation if and only if it admits an exact \({\cal C}\)-local SPSD splitting. In other words, once the cover \({\cal C}\) has been chosen, the existence of a local Gram is equivalent to the existence of local SPSD quadratic energies supported on that same cover.

The second step in the above proof shows how to construct \emph{one} particular SPSD splitting given a Gram.  
This splitting is not unique, however, with the following generalizing this construction to a family of local SPSD splittings given a local Gram. 

\begin{lemma}
\label{lem:row_atom_allocation}
Let \(A=G^{\top}G\) with \(G\in\mathbb R^{m\times n}\) ${\cal C}$-local.
Given weights \(w_{j\alpha}\ge 0\) such that, for each \(j\),
\[
\sum_{\alpha\in I} w_{j\alpha}=1,
\qquad\text{and}\qquad
w_{j\alpha}>0 \ \Longrightarrow\ \operatorname{supp}(\bm g_j)\subseteq U_\alpha .
\]
Define for each \(\alpha\in I\)
\[
\wh{\bm g}_{j\alpha}:=R_\alpha \bm g_j\in \mathbb R^{|U_\alpha|},
\qquad
\widetilde A_\alpha
:=
\sum_{j=1}^{m} w_{j\alpha}\,\wh{\bm g}_{j\alpha}\wh{\bm g}_{j\alpha}^{\top}\ \in\ \mathbb R^{|U_\alpha|\times |U_\alpha|}.
\]
Then each \(\widetilde A_\alpha\succeq 0\) and
\[
A=\sum_{\alpha\in I} R_\alpha^{\top}\widetilde A_\alpha R_\alpha,
\]
i.e.\ the weights \(\{w_{j\alpha}\}\) induce a \({\cal C}\)-local SPSD splitting of \(A\).
\end{lemma}

\begin{proof}
Since \(A=G^{\top}G=\sum_{j=1}^{m}\bm g_j\bm g_j^{\top}\) and \(\sum_{\alpha\in I} w_{j\alpha}=1\), we may write
\[
A=\sum_{j=1}^{m}\sum_{\alpha\in I} w_{j\alpha}\,\bm g_j\bm g_j^{\top}.
\]
Fix \(j\) and \(\alpha\) with \(w_{j\alpha}>0\). By assumption \(\operatorname{supp}(\bm g_j)\subseteq U_\alpha\), hence restricting \(\bm g_j\) to \(U_\alpha\)
and embedding back recovers \(\bm g_j\), i.e.\ \(\bm g_j=R_\alpha^\top\wh{\bm g}_{j\alpha}\) with \(\wh{\bm g}_{j\alpha}=R_\alpha\bm g_j\).
Therefore
\[
w_{j\alpha}\,\bm g_j\bm g_j^{\top}
=
w_{j\alpha}\,R_\alpha^{\top}\wh{\bm g}_{j\alpha}\wh{\bm g}_{j\alpha}^{\top}R_\alpha
=
R_\alpha^{\top}\bigl((\sqrt{w_{j\alpha}}\wh{\bm g}_{j\alpha})(\sqrt{w_{j\alpha}}\wh{\bm g}_{j\alpha})^{\top}\bigr)R_\alpha.
\]
Summing over \(j\) and regrouping by \(\alpha\) yields
\(A=\sum_{\alpha\in I} R_\alpha^{\top}\widetilde A_\alpha R_\alpha\).
Each \(\widetilde A_\alpha\) is a sum of rank-at-most-one SPSD matrices and is therefore SPSD.
\end{proof}

\begin{remark}
If \(w_{j\alpha}\in\{0,1\}\) and \(\sum_{\alpha} w_{j\alpha}=1\), then each row outer product \(\bm{g}_j \bm{g}_j^{\top}\) is assigned to exactly one patch \(U_\alpha\).
This is the ``hard assignment'' used in the proof of \cref{thm:local_split_iff_local_gram}.
If \(w_{j\alpha}=1/|I_j|\) for \(\alpha\in I_j\), and \(w_{j\alpha}=0\) otherwise, then each row outer product is split uniformly among all cover members that can accommodate it
(i.e.\ among all \(\alpha\) with \(\operatorname{supp}(\bm{g}_j)\subseteq U_\alpha\)).
In LS-AMG-DD \cite{southworth2026lsamgdd}, one instantiates the cover \(U_i:=\Omega_i\), and constructs an exact overlapping SPSD splitting directly from a global Gram \(A=G^{\top}G\) by distributing row outer products across patches with the diagonal multiplicity matrix \(W_{\mathfrak R_{\omega_i}}\) from \eqref{eq:Wmatrix-def}. 
Since there is one patch \(\Omega_i\) per aggregate \(\omega_i\), and row \(j\) is assigned uniformly over precisely those indices \(i\) for which it touches \(\omega_i\), this is a uniform row outer product allocation over the admissible patches in that construction.
The resulting patch matrices are SPSD and satisfy \(A=\sum_i R_{\Omega_i}^{\top}\widetilde A_i R_{\Omega_i}\), i.e.\ they fit exactly into the
row outer product allocation framework above.
\end{remark}

\section{Conforming finite-element discretizations} \label{sec:conforming_discr}

For SPD PDE discretizations, the most straightforward source of \cref{thm:local_split_iff_local_gram} is a method that is assembled from loops of local SPD or SPSD quadratic contributions. This is readily the case for conforming finite-element methods, in which the finite-element space is set to be a proper subspace of the Sobolev space occurring in the problem's weak formulation. In this case, the global matrix is assembled from cell-wise contributions
\begin{align} \label{eq:A_elem_SPDS_split}
A=\sum_{T \in {\cal T}_h} R_T^{\top}A_T R_T,
\end{align}
with local element matrices \(A_T\succeq0\). Thus the local quadratic energies and the associated cover are built into the discretization from the outset. 

In the numerical results section below, we consider SPD problems given weakly by: find $u \in \mathbb{V}({\cal D})$ such that
\[
 a(u,v) = \ell(v) \qquad
    \forall v\in \mathbb{V}({\cal D}),
\]
for Sobolev space $\mathbb{V}({\cal D})$ over a domain $\mathcal{D}$. Here $a$ is a bilinear form and $\ell$ a linear form. 
For a suitable conforming finite-element space \(V_h \subset \mathbb{V}({\cal D})\), the global bilinear form has the elementwise decomposition
\[
a(u,v)
=
\sum_{T\in{\cal T}_h} a_T(u,v),
\]
where \(a_T(\cdot,\cdot)\) denotes the contribution associated with the element \(T\), including any boundary-facet terms assigned to \(T\) (see also \cref{sec:elemt_gram_stacking}).
Equivalently, each element contributes the local quadratic energy
\(
a_T(u,u). 
\)
This local energy is the object that gives the elementwise SPSD contribution to the assembled operator.
If \(\{\phi_{T,i}\}_i\) denotes the local basis on \(T\), where the basis functions may be scalar- or vector-valued depending on the space, the element matrix is
\(
(A_T)_{ij}
=
a_T(\phi_{T,j},\phi_{T,i}).
\)

In \Cref{sec:numerics} we consider three scenarios, based on the Sobolev spaces \(H(\div;\mathcal D)\), \(H^2(\mathcal D)\), and $[H^1(\mathcal D)]^2$, together with $a$ corresponding to a grad-div plus mass, an anisotropic hyperviscosity, and a linear elasticity-type operator. Whether other discretizations admit the required local quadratic energy structure is generally less transparent. It does not necessarily hold for non-conforming discretizations such as the discontinuous Galerkin (DG) space-based interior penalty method for elliptic problems. In particular, their element assembly includes a loop over interior facet integrals, which by themselves will generally not be coercive. It is unclear if suitably reformulated versions of DG schemes may circumvent this problem, and in this work, we restrict ourselves to conforming discretizations.
In terms of finite-difference discretizations, structure-preserving finite-difference schemes designed from discrete energy or flux principles likely admit the required structure \cite{morel1998local,BeiraoDaVeiga2014MFD}. In some preliminary work not presented in this paper, we have also been able to manufacture the required structure for certain stencil-based finite-difference discretizations, essentially by introducing specific cell-local functionals and reorganizing the stencil into a cell-based SPSD form so that it more closely resembles a conforming finite-element discretization. However, it is not clear if general stencil-based finite-difference schemes admit the required structure, since from the stencil form alone, there is no canonical cover, no explicit family of local SPSD pieces, and hence no immediately visible route to \cref{thm:local_split_iff_local_gram}.

\section{Elementwise Gram formulation for conforming discretizations}
\label{sec:elementwise_gram_formulation}

In this section we construct a global Gram factor $G$ in the conforming finite-element context using element information.
We then discuss the numerical computation of elementwise Gram factors in \cref{sec:gram-local-compute}, the subtlety of potential versus assembled distance-one neighbors in $A$ in \cref{sec:agg_closure}, and then mesh-entity structure that arises in the algebraic overlaps $\{ \Omega_i \}$ when $G$ is built from elementwise information.

\subsection{Elementwise SPSD assembly and Gram stacking}
\label{sec:elemt_gram_stacking}

Let \(V_h\subset V\) be any conforming finite-element space, and let \(A\in\mathbb R^{n\times n}\) be the assembled SPD matrix.
Accordingly, we assume that \(A\) is assembled from a family of element-local SPSD contributions over the mesh as in \eqref{eq:A_elem_SPDS_split}; that is,
\begin{align} \label{eq:conforming_local_spsd_assembly}
A=\sum_{T\in\mathcal T_h} R_T^{\top}A_T R_T,
\qquad
A_T\succeq 0.
\end{align}
Recall that \(R_T:\mathbb R^n\to\mathbb R^{n_T}\) is the Boolean restriction from the global DOFs to the DOF set of element $T$, and \(A_T\in\mathbb R^{n_T\times n_T}\) is the block on element $T$.

\begin{proposition}[Elementwise Gram formulation for conforming assemblies]
\label{prop:conforming_elementwise_gram}
Assume \eqref{eq:conforming_local_spsd_assembly}. For each element \(T\in{\cal T}_h\), let \(G_T\in\mathbb R^{m_T\times n_T}\) be any local Gram factor such that
\(
\label{eq:conforming_local_gram}
A_T = G_T^{\top}G_T.
\)
If $G$ is defined by stacking the blocks \(G_T R_T\in\mathbb R^{m_T\times n}\) as
\[
G :=
\begin{bmatrix}
G_{T_1} R_{T_1}\\
G_{T_2} R_{T_2}\\
\vdots
\end{bmatrix}
\in \mathbb{R}^{m \times n},
\]
with $m = \sum_{T \in {\cal T}_h} m_T$.
Then
\begin{equation}
\label{eq:conforming_global_gram}
A = G^{\top}G,
\qquad
\mathrm{and}  
\qquad
\bm u^{\top}A\bm u
=
\sum_{T\in{\cal T}_h} \|G_T (R_{T}\bm u)\|_2^2
\quad
\forall \bm u\in\mathbb R^n.
\end{equation}
\end{proposition}

\begin{proof}
    This is simply the specialization of the construction used in the proof of \cref{thm:local_split_iff_local_gram} to cover members given by mesh elements.
\end{proof}

\cref{fig:Q1_G_split} gives a schematic interpretation of \cref{prop:conforming_elementwise_gram}. 
Each element \(T\) contributes a local Gram factor \(G_T\), and the global factor \(G\) is obtained by stacking the embedded element blocks \(G_T R_T\). 
Thus each row of \(G\) is associated with a mesh element, and the rows altogether represent an unassembled form of the conforming discretization. 
By contrast, the assembled operator \(A=\sum_T R_T^\top A_T R_T = G^\top G\) acts on the conforming global DOFs after local element copies across shared interfaces have been reduced, so individual elements are no longer separated in \(A\). 

\begin{figure}[t!]
\centering
\begin{tikzpicture}[
  >={Latex[length=2mm]},
  mesh/.style={line width=0.5pt, black!60},
  faintvtx/.style={circle, draw=black!35, fill=black!15, inner sep=1.2pt, line width=0.2pt},
  asmvtx/.style={circle, draw=black, fill=black!70, inner sep=1.6pt, line width=0.25pt},
  locA/.style={circle, draw=black, fill=blue!65, inner sep=1.6pt, line width=0.25pt},
  locB/.style={circle, draw=black, fill=orange!85!black, inner sep=1.6pt, line width=0.25pt},
  locC/.style={circle, draw=black, fill=green!55!black, inner sep=1.6pt, line width=0.25pt},
  locD/.style={circle, draw=black, fill=purple!65, inner sep=1.6pt, line width=0.25pt},
  elA/.style={fill=blue!8, draw=none},
  elB/.style={fill=orange!10, draw=none},
  elC/.style={fill=green!8, draw=none},
  elD/.style={fill=purple!8, draw=none},
  stackA/.style={fill=blue!10, draw=black, line width=0.55pt, rounded corners=1pt},
  stackB/.style={fill=orange!10, draw=black, line width=0.55pt, rounded corners=1pt},
  stackC/.style={fill=green!10, draw=black, line width=0.55pt, rounded corners=1pt},
  stackD/.style={fill=purple!10, draw=black, line width=0.55pt, rounded corners=1pt},
  panelcap/.style={font=\small, align=center},
  callout/.style={draw=black!75, dashed, line width=0.7pt}
]

\def\s{1.4}
\def\e{0.175}
\def\gapAB{1.65}
\def\gapBC{1.65}
\def\stackW{1.95}
\def\blockH{0.68}
\def\ycap{-0.48}

\pgfmathsetmacro{\twos}{2*\s}
\pgfmathsetmacro{\stackH}{4*\blockH}
\pgfmathsetmacro{\xA}{0}
\pgfmathsetmacro{\xB}{\twos+\gapAB}
\pgfmathsetmacro{\xC}{\xB+\stackW+\gapBC}
\pgfmathsetmacro{\yone}{\blockH}
\pgfmathsetmacro{\ytwo}{2*\blockH}
\pgfmathsetmacro{\ythree}{3*\blockH}
\pgfmathsetmacro{\yfour}{4*\blockH}

\begin{scope}[shift={(\xA,0)}]
  \path[elA] (0,0) rectangle (\s,\s);
  \path[elB] (\s,0) rectangle (\twos,\s);
  \path[elC] (0,\s) rectangle (\s,\twos);
  \path[elD] (\s,\s) rectangle (\twos,\twos);

  \foreach \k in {0,\s,\twos} {
    \draw[mesh] (\k,0) -- (\k,\twos);
    \draw[mesh] (0,\k) -- (\twos,\k);
  }
  \draw[mesh] (0,0) rectangle (\twos,\twos);

  \foreach \x in {0,\s,\twos} {
    \foreach \y in {0,\s,\twos} {
      \node[faintvtx] at (\x,\y) {};
    }
  }

  \node[locA] at (\e,\e) {};
  \node[locA] at ({\s-\e},\e) {};
  \node[locA] at (\e,{\s-\e}) {};
  \node[locA] at ({\s-\e},{\s-\e}) {};

  \node[locB] at ({\s+\e},\e) {};
  \node[locB] at ({\twos-\e},\e) {};
  \node[locB] at ({\s+\e},{\s-\e}) {};
  \node[locB] at ({\twos-\e},{\s-\e}) {};

  \node[locC] at (\e,{\s+\e}) {};
  \node[locC] at ({\s-\e},{\s+\e}) {};
  \node[locC] at (\e,{\twos-\e}) {};
  \node[locC] at ({\s-\e},{\twos-\e}) {};

  \node[locD] at ({\s+\e},{\s+\e}) {};
  \node[locD] at ({\twos-\e},{\s+\e}) {};
  \node[locD] at ({\s+\e},{\twos-\e}) {};
  \node[locD] at ({\twos-\e},{\twos-\e}) {};

  \node[font=\scriptsize] at ({0.5*\s},{0.5*\s}) {\(T_1\)};
  \node[font=\scriptsize] at ({1.5*\s},{0.5*\s}) {\(T_2\)};
  \node[font=\scriptsize] at ({0.5*\s},{1.5*\s}) {\(T_3\)};
  \node[font=\scriptsize] at ({1.5*\s},{1.5*\s}) {\(T_4\)};

  \draw[callout] (\s,\s) circle[radius=0.34];

  \node[panelcap] at ({\s},{\ycap}) {(a) Unassembled operator};
\end{scope}

\draw[->, line width=0.8pt] ({\twos+0.18},{\s}) -- ({\xB-0.15},{\s});

\begin{scope}[shift={(\xB,0)}]
  \draw[line width=0.7pt] (0,0) rectangle (\stackW,\stackH);

  \path[stackD] (0,\ythree) rectangle (\stackW,\yfour);
  \path[stackC] (0,\ytwo) rectangle (\stackW,\ythree);
  \path[stackB] (0,\yone) rectangle (\stackW,\ytwo);
  \path[stackA] (0,0) rectangle (\stackW,\yone);

  \foreach \yb in {0,\yone,\ytwo,\ythree} {
    \foreach \t in {0.18,0.38,0.58,0.78} {
      \draw[black!18, line width=0.25pt]
        ({0.10},{\yb+\t*\blockH}) -- ({\stackW-0.10},{\yb+\t*\blockH});
    }
  }

  \node[font=\tiny] at ({0.5*\stackW},{0.5*\yone}) {\(G_{T_1}R_{T_1}\)};
  \node[font=\tiny] at ({0.5*\stackW},{0.5*(\yone+\ytwo)}) {\(G_{T_2}R_{T_2}\)};
  \node[font=\tiny] at ({0.5*\stackW},{0.5*(\ytwo+\ythree)}) {\(G_{T_3}R_{T_3}\)};
  \node[font=\tiny] at ({0.5*\stackW},{0.5*(\ythree+\yfour)}) {\(G_{T_4}R_{T_4}\)};

  \node[panelcap] at ({0.5*\stackW},{\ycap}) {(b) Gram factor \(G\)};
\end{scope}

\draw[->, line width=0.8pt] ({\xB+\stackW+0.18},{\s}) -- ({\xC-0.15},{\s});

\begin{scope}[shift={(\xC,0)}]
  \foreach \k in {0,\s,\twos} {
    \draw[mesh] (\k,0) -- (\k,\twos);
    \draw[mesh] (0,\k) -- (\twos,\k);
  }
  \draw[mesh] (0,0) rectangle (\twos,\twos);

  \foreach \x in {0,\s,\twos} {
    \foreach \y in {0,\s,\twos} {
      \node[faintvtx] at (\x,\y) {};
    }
  }

  \draw[callout] (\s,\s) circle[radius=0.18];

  \node[panelcap] at ({\s},{\ycap}) {(c) Assembled operator \(A=G^\top G\)};
\end{scope}

\end{tikzpicture}
\caption{Element-local unassembled versus assembled perspectives in conforming \(Q_1\) finite-element discretization on a \(2\times 2\) element patch. 
Left: each element \(T\) carries its own local copies of the vertex DOFs; these are the local coordinates on which the element block \(A_T\) acts, with \(A_T=G_T^\top G_T\).
Middle: stacking the embedded element factors \(G_TR_T\) yields the global Gram factor \(G\), where each block row is associated with one element. 
Right: after assembly, the duplicated local copies are reduced to a single conforming global DOF.}
\label{fig:Q1_G_split}
\end{figure}
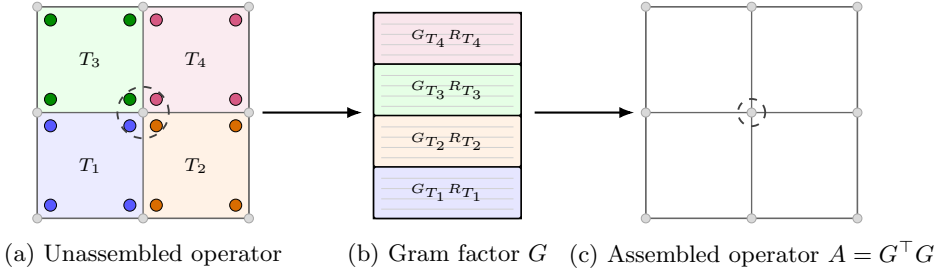

The formulation described above requires some nuance in the presence of essential boundary conditions.
If essential boundary conditions are imposed weakly by a symmetric Nitsche method, then the boundary-intersecting element matrices $\{ A_T : T \cap \partial {\cal D} \neq \emptyset \}$ acquire additional boundary-facet contributions. Provided the penalty is chosen large enough that these modified local matrices remain SPSD, \cref{prop:conforming_elementwise_gram} applies verbatim.
By contrast, strong imposition of essential conditions typically modifies the assembled global system, e.g., via eliminating boundary DOFs, which is compatible with the formulation above, but requires extra bookkeeping to express the post-elimination operator as a sum of local SPSD pieces.
To avoid such bookkeeping in our code implementation, the numerical results of \cref{sec:numerics} use discretizations with weakly enforced boundary conditions via Nitsche's method. Such discretizations can still be seen as conforming with respect to continuous formulations in which the boundary conditions are enforced weakly via Lagrange multipliers \cite{stenberg1995some}, which in turn are equivalent to continuous weak formulations with strongly enforced boundary conditions.

\begin{remark}[Finer-scale Grams induce coarser-scale Grams]
\label{rem:coarse_gram}
The local Gram construction in \cref{prop:conforming_elementwise_gram} can be generalized from cover members of individual mesh elements arbitrary collections of elements, since by \eqref{eq:conforming_local_spsd_assembly} the sum over any collection of elements is also SPSD, and hence admits a Gram.
The matrix $G \in \mathbb{R}^{m \times n}$ is generally tall-and-skinny with $m > n$, because DOFs shared across interfaces of cover members contribute one row for every member they appear in; for covers over increasingly larger collections of elements $m \to n$. 
\end{remark}

\subsection{Computing local Gram factors}
\label{sec:gram-local-compute}

\Cref{prop:conforming_elementwise_gram} reduces the global construction \(A=G^{\top}G\) to local factorizations
\(A_T=G_T^{\top}G_T\) of the SPSD blocks \(A_T\succeq 0\). 
We now briefly discuss standard \emph{algebraic} choices for \(G_T\).
In general, a direct method would be preferable over an iterative one for stability reasons.
If \(A_T\succ 0\), take a Cholesky factorization \(A_T=L_TL_T^{\top}\)
and set \(G_T:=L_T^{\top}\). This gives a square factor \(G_T\in\mathbb R^{n_T\times n_T}\).
Otherwise, if $A_T$ is only semidefinite, then one has to be more careful. For example, one can compute a pivoted Cholesky that produces a rectangular factor \(G_T \in \mathbb{R}^{m_T \times n_T} \) with \(m_T=\rank(A_T)\) in exact arithmetic, and with a numerically stable truncation in floating point.
In terms of iterative methods, regardless of whether $A_T$ is definite, an eigendecomposition can be used, with an explicit range factor obtained from 
\(A_T=Q_T\Lambda_T Q_T^{\top}\) by keeping only the positive spectrum: if \(\Lambda_{T,+}\) collects the positive eigenvalues and \(Q_{T,+}\) the corresponding eigenvectors, then \(G_T:=\Lambda_{T,+}^{1/2}Q_{T,+}^{\top}\in\mathbb R^{r_T\times n_T}\) satisfies \(A_T=G_T^{\top}G_T\),
with \(r_T=\rank(A_T)\). 
In general, observe that local Gram factors are not unique; for example, if \(A_T=G_T^{\top}G_T\) and \(U\) is orthogonal, then \((UG_T)^{\top}(UG_T)=A_T\). 
\subsection{A subtlety in defining the overlapping subdomains}
\label{sec:agg_closure}

The LS-AMG-DD subdomain construction of \eqref{eq:overlap_def} from \cref{sec:background} is based on local row information from the Gram factor \(G\), whereas a more standard Schwarz-style overlapping subdomain would be defined from the graph of the assembled matrix \(A=G^\top G\).  
These two viewpoints are closely related, but they are not equivalent in general, and as we now illustrate their difference can have a notable effect on the resulting LS-AMG-DD solver.  
To this end, we introduce the following two closures of an arbitrary aggregate \(\omega_i\).
Except for the comparison in \cref{tab:g_vs_a_closure} below, in this paper we use \cref{def:G_row_closure} to form the overlapping and interface sets, and not \cref{def:A_algebraic_closure}.

\begin{definition}[Algebraic \(A\)-closure of \(\omega_i\)]
\label{def:A_algebraic_closure}
The algebraic \(A\)-closure of \(\omega_i\) is
\[
\Omega_i^A
:=
\omega_i
\cup
\bigl\{
p\in\{1,\ldots,n\}\setminus\omega_i :
A(p,\omega_i)\neq 0
\bigr\},
\]
where \(A(p,\omega_i)\neq 0\) means that the row block \(A(p,\omega_i)\) is not identically zero.
We also define the corresponding algebraic \(A\)-interface set
\(
\Gamma_i^A := \Omega_i^A \setminus \omega_i.
\)
\end{definition}

\begin{definition}[\(G\)-row closure of \(\omega_i\)]
\label{def:G_row_closure}
Let \(\mathfrak R_{\omega_i}\) be the row-touching set of $\omega_i$ as in \eqref{eq:calR_def}.
The \(G\)-row closure of \(\omega_i\) is
\[
\Omega_i^G
:=
\bigcup_{j\in\mathfrak R_{\omega_i}}\operatorname{supp}(\bm g_j).
\]
We also define the corresponding \(G\)-row interface set
\(
\Gamma_i^G := \Omega_i^G\setminus\omega_i.
\)
\end{definition}

Practically speaking, the two closures \(\Omega_i^A\) and \(\Omega_i^G\) of $\omega_i$ are almost identical, but there is a subtle difference in terms of the assembled entries
of \(A=G^\top G\).  
For \(p,q\in\{1,\ldots,n\}\),
\(
    A(p,q)
    =
    \sum_{j=1}^m G(j,p)G(j,q).
\)
Thus, for fixed \(q\in\omega_i\), the \(A\)-closure detects those DOFs \(p\)
for which this assembled sum $A(p, q)$ is nonzero for at least one \(q\in\omega_i\).
In contrast, the \(G\)-row closure detects all DOFs \(p\) that appear together
with some \(q\in\omega_i\) in at least one row of \(G\), independently of
whether the corresponding contributions cancel in the assembled entry
\(A(p,q)\).  In this sense, \(\Omega_i^G\) records the potential distance-one
neighbors visible from rows of \(G\), whereas \(\Omega_i^A\)
records only the distance-one neighbors that remain after assembly, \(A = G^\top G\).

Both closures define reasonable overlapping subdomains for a Schwarz smoother; however, they are not both appropriate for constructing the LS-AMG-DD hierarchy. Specifically, if the \(A\)-closure of an aggregate is strictly smaller than its \(G\)-row closure, then the local matrices built on the \(A\)-closure do not
contain all of the row information used in that splitting. Consequently, the
exact SPSD splitting need not hold with the \(A\)-closure. The example in \cref{tab:g_vs_a_closure} shows that this situation can occur in practice and that the difference can have a substantial effect on the solver. We note that it is common practice to remove zero or near-zero entries from the sparsity graph of an assembled matrix $A$ \cite{anderson2021mfem,bell2023pyamg}. 

\begin{table}[h!]
\caption{Iteration counts for LS-AMG-DD applied to the scalar $H^1({\cal D})$ problem $-\div(\kappa\grad u)=0$ using either the algebraic $A$-closure in \cref{def:A_algebraic_closure} or the $G$-row closure in \cref{def:G_row_closure}. 
The discretization is CG(1) on a triangular mesh under uniform refinement. 
The spatial domain is ${\cal D}=(0,1)^2$, with $u=0$ on $\partial {\cal D}$ (imposed weakly via Nitsche's method), and the diffusion coefficient is $\kappa(x,y)=1+x+y$. 
The solver uses the same general setup as the tests in \cref{sec:numerics}, with the local eigenvalue threshold given by $\tau_{\rm cut}=\max_j {\rm mult}_{\omega}(j)$.}
\label{tab:g_vs_a_closure}
\centering
\small
\begin{tabular}{c|cccccc}
\hline
 & \multicolumn{6}{c}{DOFs} \\
closure
& 1,089 & 4,225 & 16,641 & 66,049 & 263,169 & 1,050,625 \\
\hline
\(A\)-closure     & 10 & 12 & 13 & 28 & 72 & 158 \\
\(G\)-row closure & 10 & 10 & 11 & 11 & 11 & 11  \\
\hline
\end{tabular}
\end{table}

\subsection{Geometric structure in overlapping Schwarz smoothers}
\label{sec:G_closure_geometry}

The importance of non-pointwise smoothers in multigrid methods is well known and studied, e.g. \cite{arnold2000hdivhcurlmg,claus2019saddlesmoothers,farrell2021pcpatch}, primarily in the geometric context. In this section we show that, under reasonable assumptions, if \(G\) arises from an element-local Gram construction, then the overlapping subdomains \(\{\Omega_i\}\) recover an \emph{element-induced} patch structure. Hence, a Schwarz smoother posed on the overlaps \(\{\Omega_i\}\) uses the same type of element-induced patch structure commonly used in geometric multigrid smoothers.
We first define the relevant element closure and then state the main theorem.

\begin{definition}[Element closure]
\label{def:element_closure_dofs}
For any element \(T\in\mathcal T_h\), let \(\mathcal N(T)\subseteq\{1,\dots,n\}\) denote the global DOF indices supported on \(T\). Let \(\operatorname{cl}_T(\omega):=\{\,T\in\mathcal T_h:\ \mathcal N(T)\cap\omega\neq\emptyset\,\}\) be the \emph{element closure} of any DOF set \(\omega\subseteq\{1,\dots,n\}\), and \(\Omega_{\mathrm{cl}}(\omega):=\bigcup_{T\in\operatorname{cl}_T(\omega)} \mathcal N(T)\subseteq\{1,\dots,n\}\) be the associated \emph{element-closure DOF set}.
\end{definition}

\begin{assumption}[Full-support local row per element]
\label{ass:element_full_support_row}
For each element \(T\in\mathcal T_h\), let \(A_T\in\mathbb R^{n_T\times n_T}\) be the local block from \cref{prop:conforming_elementwise_gram}, and let \(G_T\in\mathbb R^{m_T\times n_T}\) satisfy \(A_T=G_T^{\top}G_T\). Let \(\{\bm g_{T,j}\}_{j=1}^{m_T}\subseteq \mathbb R^{n_T}\) be such that the rows of \(G_T\) are \(\bm g_{T,j}^\top\).
We assume that for every \(T\in\mathcal T_h\), there exists an index \(q_T\in\{1,\dots,m_T\}\) such that \(\operatorname{supp}(\bm g_{T,q_T})=\{1,\dots,n_T\}\).
\end{assumption}

In words, \cref{ass:element_full_support_row} says that for each element \(T\), the chosen local Gram factor \(G_T\) has at least one row that is nonzero in every local DOF on \(T\).
Under this hypothesis, the next theorem shows that the algebraic \(G\)-row closure of \cref{def:G_row_closure} agrees with the geometric element closure of \cref{def:element_closure_dofs}.
A schematic example is shown in \cref{fig:triangle_hcurl_closure}, where a small aggregate DOF set \(\omega\) on a triangular \(\hcurl\) discretization enlarges to the element-closure DOF set \(\Omega_{\mathrm{cl}}(\omega)\), which under \cref{ass:element_full_support_row} coincides with \(\Omega^G(\omega)\).\footnote{\Cref{ass:element_full_support_row} does not hold automatically in all cases, but it can typically be enforced numerically. Suppose it does not hold for $G_T$, then one can typically find an orthogonal matrix $U$ such that it does hold for \(\wh{G}_T = U G_T\), noting that $\wh{G}_T^\top \wh{G}_T = (U G_T)^\top (U G_T) = G_T^\top G_T$.}

\begin{figure}[t!]
\centering
\begin{tikzpicture}[scale=1.2, line cap=round, line join=round]

\colorlet{ambientdof}{black}
\colorlet{aggdof}{red!75!black}
\colorlet{closedof}{blue!70!black}

\def\FilledDot#1{\fill[ambientdof] #1 circle (1.15pt);}
\def\OpenDot#1{\draw[closedof, line width=0.55pt, fill=white] #1 circle (1.5pt);}
\def\XMark#1{%
  \begin{scope}[shift={#1}]
    \draw[aggdof, line width=0.8pt] (-0.085,-0.085) -- (0.085,0.085);
    \draw[aggdof, line width=0.8pt] (-0.085,0.085) -- (0.085,-0.085);
  \end{scope}}

\def\EdgeDOFs#1#2#3{%
  #3{($(#1)!0.32!(#2)$)}
  #3{($(#1)!0.68!(#2)$)}}

\def\InteriorDOFs#1#2#3#4{%
  \begin{scope}[shift={(barycentric cs:#1=1,#2=1,#3=1)}]
    #4{(-0.11,0)}
    #4{(0.11,0)}
  \end{scope}}

\def\TriDOFs#1#2#3#4{%
  \EdgeDOFs{#1}{#2}{#4}
  \EdgeDOFs{#2}{#3}{#4}
  \EdgeDOFs{#3}{#1}{#4}
  \InteriorDOFs{#1}{#2}{#3}{#4}}

\def\SetMeshCoords{%
  \coordinate (A) at (0.0,0.0);
  \coordinate (B) at (1.5,0.0);
  \coordinate (C) at (3.0,0.0);
  \coordinate (D) at (0.0,1.3);
  \coordinate (E) at (1.5,1.3);
  \coordinate (F) at (3.0,1.3);
  \coordinate (G) at (0.0,2.6);
  \coordinate (H) at (1.5,2.6);
  \coordinate (I) at (3.0,2.6);}

\def\DrawMesh{%
  \draw[line width=0.4pt] (A)--(B)--(C);
  \draw[line width=0.4pt] (D)--(E)--(F);
  \draw[line width=0.4pt] (G)--(H)--(I);
  \draw[line width=0.4pt] (A)--(D)--(G);
  \draw[line width=0.4pt] (B)--(E)--(H);
  \draw[line width=0.4pt] (C)--(F)--(I);
  \draw[line width=0.4pt] (A)--(E);
  \draw[line width=0.4pt] (B)--(F);
  \draw[line width=0.4pt] (D)--(H);
  \draw[line width=0.4pt] (E)--(I);}

\def\ShadeTouchedElementsB{%
  \fill[aggdof!10] (A)--(B)--(E)--cycle;
  \fill[aggdof!10] (B)--(F)--(E)--cycle;}
\def\ShadeTouchedElementsC{%
  \fill[closedof!10] (A)--(B)--(E)--cycle;
  \fill[closedof!10] (B)--(F)--(E)--cycle;}

\begin{scope}[shift={(0,0)}]
  \SetMeshCoords
  \DrawMesh

  \TriDOFs{A}{B}{E}{\FilledDot}
  \TriDOFs{A}{E}{D}{\FilledDot}
  \TriDOFs{B}{C}{F}{\FilledDot}
  \TriDOFs{B}{F}{E}{\FilledDot}
  \TriDOFs{D}{E}{H}{\FilledDot}
  \TriDOFs{D}{H}{G}{\FilledDot}
  \TriDOFs{E}{F}{I}{\FilledDot}
  \TriDOFs{E}{I}{H}{\FilledDot}

  \node[font=\footnotesize] at (1.5,-0.42) {(a) ambient local DOF layout};
\end{scope}

\begin{scope}[shift={(3.8,0)}]
  \SetMeshCoords
  \ShadeTouchedElementsB
  \DrawMesh

  \EdgeDOFs{B}{E}{\XMark}

  \node[font=\footnotesize] at (1.5,-0.42) {(b) aggregate DOF set \(\omega\)};
\end{scope}

\begin{scope}[shift={(7.6,0)}]
  \SetMeshCoords
  \ShadeTouchedElementsC
  \DrawMesh

  \EdgeDOFs{A}{B}{\OpenDot}
  \EdgeDOFs{B}{E}{\OpenDot}
  \EdgeDOFs{E}{A}{\OpenDot}
  \EdgeDOFs{B}{F}{\OpenDot}
  \EdgeDOFs{F}{E}{\OpenDot}
  \InteriorDOFs{A}{B}{E}{\OpenDot}
  \InteriorDOFs{B}{F}{E}{\OpenDot}

  \node[font=\footnotesize] at (1.5,-0.42) {(c) \(\Omega_{\mathrm{cl}}(\omega)=\Omega^G(\omega)\)};
\end{scope}

\end{tikzpicture}
\caption{Illustration of \cref{thm:elem_closure} for a conforming \(\hcurl\) discretization using higher-order first-family N\'ed\'elec elements. Each triangular element consists of two tangential moments on each edge and two cell-interior DOFs. Filled black circles represent the local DOF layout, red crosses represent the chosen aggregate DOF set \(\omega\), and open blue circles represent the resulting closure. In panel (b), the aggregate is chosen as the two edge-based DOFs supported on a single interior mesh edge; the lightly shaded triangles are the touched-element set \(\operatorname{cl}_T(\omega)\). Panel (c) shows the associated element-closure DOF set \(\Omega_{\mathrm{cl}}(\omega)\), obtained by taking all DOFs supported on those touched elements. 
}
\label{fig:triangle_hcurl_closure}
\vspace{-4ex}
\end{figure}
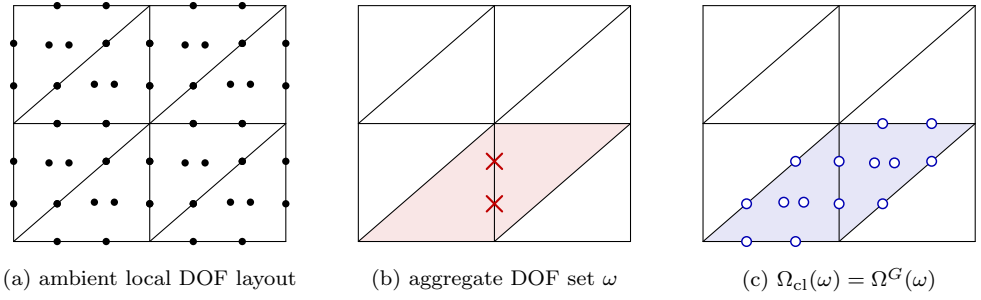

\begin{theorem}
\label{thm:elem_closure}
Let \(A=G^{\top}G\) be assembled from the blocks \(G_TR_T\), \(T\in\mathcal T_h\) as in \cref{prop:conforming_elementwise_gram}. For any DOF set \(\omega\subset\{1,\dots,n\}\), let \(\Omega^G(\omega)\) denote its \(G\)-row closure from \cref{def:G_row_closure}, and let \(\Omega_{\mathrm{cl}}(\omega)\) denote its element-closure DOF set from \cref{def:element_closure_dofs}. Then \(\Omega^G(\omega)\subseteq \Omega_{\mathrm{cl}}(\omega)\). If, in addition, \cref{ass:element_full_support_row} holds, then \(\Omega^G(\omega)=\Omega_{\mathrm{cl}}(\omega)\).
\end{theorem}

\begin{proof}

Fix a DOF set \(\omega\subset\{1,\dots,n\}\).

We first show that \(\Omega^G(\omega)\subseteq \Omega_{\mathrm{cl}}(\omega)\). Consider a DOF \(j\in\Omega^G(\omega)\). By definition of the \(G\)-row closure, there exists a row index \(r\in\mathcal R_\omega\) such that \(j\in\operatorname{supp}(\bm g_r)\). Since \(G\) is obtained by stacking the element blocks \(\wh G_T:=G_TR_T\), the row indexed by \(r\) belongs to some element block \(\wh G_T\), where \(T\in\mathcal T_h\). Every row of \(\wh G_T\) is of the form \(\bm g_{T,\ell}^\top R_T\), for vectors \(\bm g_{T,\ell} \in \mathbb{R}^{n_T}\), and is therefore supported in \(\mathcal N(T)\). Hence \(j\in\mathcal N(T)\). 
Moreover, because \(r\in\mathcal R_\omega\), the row indexed by \(r\) touches \(\omega\), so there exists a DOF \(q\in\omega\cap\operatorname{supp}(\bm g_r)\subseteq\omega\cap\mathcal N(T)\). Thus \(\mathcal N(T)\cap\omega\neq\emptyset\), which means that \(T\in\operatorname{cl}_T(\omega)\) by \cref{def:element_closure_dofs}. Therefore \(j\in\mathcal N(T)\subseteq\Omega_{\mathrm{cl}}(\omega)\).

Now assume \cref{ass:element_full_support_row}. We show that \(\Omega_{\mathrm{cl}}(\omega)\subseteq\Omega^G(\omega)\). Consider a DOF \(j\in\Omega_{\mathrm{cl}}(\omega)\). Then \(j\in\mathcal N(T)\) for some element \(T\in\operatorname{cl}_T(\omega)\). Since \(T\in\operatorname{cl}_T(\omega)\), there exists a DOF \(q\in\omega\cap\mathcal N(T)\). By \cref{ass:element_full_support_row}, there exists a row index \(q_T\in\{1,\dots,m_T\}\) such that the row \(\bm g_{T,q_T}^\top\) of \(G_T\) satisfies \(\operatorname{supp}(\bm g_{T,q_T})=\{1,\dots,n_T\}\). Hence the corresponding row \(\bm g_{T,q_T}^\top R_T\) of the element block \(\wh G_T\) has global support exactly \(\mathcal N(T)\). In particular, this row contains both the DOF \(q\in\omega\) and the DOF \(j\in\mathcal N(T)\). Its global row index therefore lies in \(\mathcal R_\omega\), and so \(j\in\Omega^G(\omega)\).
Thus, combining the unconditional inclusion \(\Omega^G(\omega)\subseteq \Omega_{\mathrm{cl}}(\omega)\) with the conditional inclusion \(\Omega_{\mathrm{cl}}(\omega)\subseteq\Omega^G(\omega)\) yields \(\Omega^G(\omega)=\Omega_{\mathrm{cl}}(\omega)\) when \cref{ass:element_full_support_row} holds.
\end{proof}

\section{Numerical results}
\label{sec:numerics}

This section tests the LS-AMG-DD solver on several conforming finite-element discretizations, which are generally known to be challenging or intractable for classical AMG methods. The finite elements are implemented in Firedrake \cite{FiredrakeUserManual}, with the global Gram factor $G$ constructed using broken finite-element spaces for practical reasons, as detailed in \cref{subsec:numerical_gram_construction}. 
All tests use standalone LS-AMG-DD V-cycles with forward
multiplicative Schwarz pre-smoothing and reverse-order post-smoothing.\footnote{RAS smoothing could be used as in \cite{southworth2026lsamgdd}, which we find often yields comparable multilevel convergence; however, there are cases where RAS is not a convergent smoother \cite{krzysik2026theory}, so we avoid it here for simplicity.}  %
The initial iterate is a random vector, and the iteration is stopped when the residual vector in the \(\ell^2\) norm has been reduced by \(10^{-10}\) relative to its initial value.
All tests use two passes of standard aggregation to build the disjoint aggregates \eqref{eq:agg_def}, and they use the adaptive local eigenvalue threshold \(\tau_{\rm cut} \propto \max_j {\rm mult}_{\omega}(j)\) proposed in \eqref{eq:tau-cut}, with a constant of proportionality that we vary slightly between different problems. In general, we find that as this constant of proportionality varies slightly around one, neither the operator complexity nor coarsening ratio of the solver varies dramatically, but slightly decreasing it can lead to a marked improvement in iteration counts, e.g., by a factor of two.
The maximum number of levels in the hierarchy is capped at three, and we coarsen until there are fewer than 10 DOFs or the nonzero density of the current operator exceeds 25\%.  
In the tests we report the operator complexity of the multilevel hierarchy, defined as
\(
    \sum_{\ell}\operatorname{nnz}(A_\ell) /
         \operatorname{nnz}(A_0),
\)
where \(A_\ell\) denotes the operator on level \(\ell \in \{0, 1, 2\}\).
The implementation is built on PyAMG \cite{bell2023pyamg}.

In
\cref{subsec:hdiv_graddiv_problem,subsec:h2_hyperdiffusion_problem,subsec:h1_vector_linear_elasticity_problem} we report iteration counts and operator complexities for LS-AMG-DD applied to \(H(\div)\) grad--div problems, scalar \(H^2\)
anisotropic hyperdiffusion problems, and vector \(H^1\) elasticity problems. We present performance across mesh refinement, finite-element order, and relevant physical parameters of reaction constant, anisotropy ratio, and Poisson ratio, respectively. To compare against a standard AMG baseline, we also show comparisons against two-level classical Ruge--Stuben (RS) AMG, and root-node (RN) AMG \cite{manteuffel2017root} on the homogeneous algebraic problem, \(A \bm x= \bm 0\).
Both AMG baselines use the PyAMG implementations, with the RN AMG solver using the default PyAMG settings, except that we use a classical strength-of-connection measure with threshold $\theta = 0.5$.
The RS AMG solver uses the same strength measure as RN, classical interpolation, second-pass RS coarsening, forward Gauss--Seidel pre-smoothing, and backward Gauss--Seidel post-smoothing.  
These results are presented to emphasize that each problem we chose is indeed hard or intractable for strong black-box AMG methods. Although for each class of problem there exist specialized, problem-dependent AMG or multigrid solvers that may perform better than black-box AMG, they would almost certainly not generalize across problem class, and it is unclear if they would be robust to the variations we consider within each class, such as a 4th-order operator with high-order elements and anisotropy, or elasticity with large Poisson ratio and high order elements. 

\begin{table}[t!]
\centering
\small
\begin{tabular}{llrrrrrrrr}
\hline
PDE & Regime & \(n_0\) & \(n_1\) & \(n_2\) &
\(n_0/n_1\) & \(n_1/n_2\) & OC & GC & Its. \\
\hline
\(\hdiv\) & easy &
\(280{,}800\) & \(48{,}312\) & \(6{,}085\) &
\(5.8\) & \(7.9\) & \(4.5\) & \(1.2\) & \(14\) \\
\(\hdiv\) & hard &
\(280{,}800\) & \(43{,}956\) & \(5{,}879\) &
\(6.4\) & \(7.5\) & \(4.1\) & \(1.2\) & \(14\) \\
\hline
\(H^2\) & easy &
\(242{,}406\) & \(31{,}062\) & \(2{,}921\) &
\(7.8\) & \(10.6\) & \(2.5\) & \(1.1\) & \(19\) \\
\(H^2\) & hard &
\(242{,}406\) & \(46{,}653\) & \(5{,}224\) &
\(5.2\) & \(8.9\) & \(4.8\) & \(1.2\) & \(24\) \\
\hline
vector \(H^1\) & easy &
\(181{,}502\) & \(37{,}474\) & \(3{,}631\) &
\(4.8\) & \(10.3\) & \(6.0\) & \(1.2\) & \(12\) \\
vector \(H^1\) & hard &
\(181{,}502\) & \(52{,}406\) & \(5{,}127\) &
\(3.5\) & \(10.2\) & \(10.8\) & \(1.3\) & \(12\) \\
\hline
\end{tabular}
\caption{For representative LS-AMG-DD hierarchies, the table reports problem sizes on levels 0, 1, and 2, coarsening ratios, operator complexities (OC), grid complexities (GC), with GC defined as $(n_0 + n_1 + n_2)/n_0$, and iteration counts (Its.) 
For each PDE type the parameter regime is varied from ``easy'' to ``hard.''
The \(\hdiv\) cases use BDM elements of degree \(3\), with grad--div weight \(\alpha=10^{-3}\) in the easy
regime and \(\alpha=10^3\) in the hard regime.  
The \(H^2\) cases use Bell elements and the S-curve field, with anisotropy ratio $\epsilon = 1$ in the easy regime and $\epsilon = 10^{-6}$ in the hard regime.
The vector \(H^1\) cases use degree \(3\) vector CG elements, with the easy and hard regimes corresponding to Poisson ratios of 
\(\lambda/\mu=1\) and \(\lambda/\mu=499\), respectively.}
\label{tab:lsdd_coarsening_ratios}
\end{table}

While the following three subsections provide detailed problem discussion and results, we first present \cref{tab:lsdd_coarsening_ratios}, showing level-specific problem sizes and coarsening ratios for easier and harder versions of the three PDE problems considered. 
The local spectral tolerance enforces a certain coarse-grid quality on all levels and robust convergence, which generally drives higher operator complexities in harder realizations of the problems considered. Nevertheless, the overall hierarchies and iteration counts remain fairly similar across problem size and type, despite significant changes in structure and difficulty of the underlying problems. We acknowledge that the coarsest grids are relatively large, but point out that dense linear algebra on modern architectures is extremely fast, and due to significant fill-in, further coarsening is not advantageous from a performance perspective. Moreover, almost all of these problems are shown to be intractable for standard black-box AMG solvers, even as two-level methods.

\subsection{Vector \texorpdfstring{\(\hdiv\)}{H(div)}: grad--div}
\label{subsec:hdiv_graddiv_problem}

First we consider an SPD vector problem in \(H(\div;\mathcal D)\).
Let \(\mathcal D=(0,1)^d\) with \(d=2\), and define
\[
    H(\div;\mathcal D)
    =
    \{\bm v\in L^2(\mathcal D)^d:\div\bm v\in L^2(\mathcal D)\}.
\]
For constants \(\alpha, \beta>0\), the strong form of the problem we consider is
\begin{equation}
\label{eq:hdiv_strong_graddiv_problem}
    -\grad(\alpha\div\bm u)+\beta\bm u=\bm f
    \quad\text{in }\mathcal D,
    \qquad
    \bm u\cdot\bm n=0
    \quad\text{on }\partial\mathcal D .
\end{equation}
Testing \eqref{eq:hdiv_strong_graddiv_problem} with
\(\bm v\in H(\div;\mathcal D)\) and integrating by parts gives
\[
    \alpha(\div\bm u,\div\bm v)_\mathcal D
    +
    \beta(\bm u,\bm v)_\mathcal D
    -
    \alpha \langle \div\bm u,\bm v\cdot\bm n\rangle_{\partial\mathcal D}
    =
    (\bm f,\bm v)_\mathcal D .
\]
Here, \(\langle \cdot,\cdot\rangle_{\partial\mathcal D}\) denotes the
\(L^2(\partial\mathcal D)\) integral
\(\langle r,s\rangle_{\partial\mathcal D}
=\int_{\partial\mathcal D}rs\,ds\).
As discussed in \cref{sec:elemt_gram_stacking}, essential boundary conditions are imposed weakly in all of our tests. 
Thus, the discrete trial and test functions are taken from the unrestricted \(H(\div;\mathcal D)\)-conforming finite-element space \(V_h\) and the bilinear form is augmented by a symmetric Nitsche boundary penalty formulation corresponding to weakly enforced boundary conditions. 
Specifically, the discretized problem is: find \(\bm u_h\in V_h\) such that
\(
    a_h(\bm u_h,\bm v_h)=\ell_h(\bm v_h)
    \;
    \forall \bm v_h\in V_h,
\)
with \(\ell_h(\bm v_h)=(\bm f,\bm v_h)_\mathcal D\) and
\begin{align}
\label{eq:hdiv_weak_penalty_form}
    a_h(\bm u_h,\bm v_h)
    :={}&
    \alpha(\div\bm u_h,\div\bm v_h)_\mathcal D
    +
    \beta(\bm u_h,\bm v_h)_\mathcal D
    +
    \langle \gamma\,\bm u_h\cdot\bm n,\bm v_h\cdot\bm n\rangle_{\partial\mathcal D}
    \nonumber\\
    &-
    \alpha \langle \div\bm u_h,\bm v_h\cdot\bm n\rangle_{\partial\mathcal D}
    -
    \alpha \langle \div\bm v_h,\bm u_h\cdot\bm n\rangle_{\partial\mathcal D}.
\end{align}
On each boundary facet \(F\), with diameter \(h_F\), we take
\(
    \gamma_F=\frac{16 \alpha p^2}{h_F},
\)
where \(p\) is the polynomial degree of the conforming finite element.

\begin{figure}[t!]
    \centering
    \includegraphics[width=0.95\linewidth]{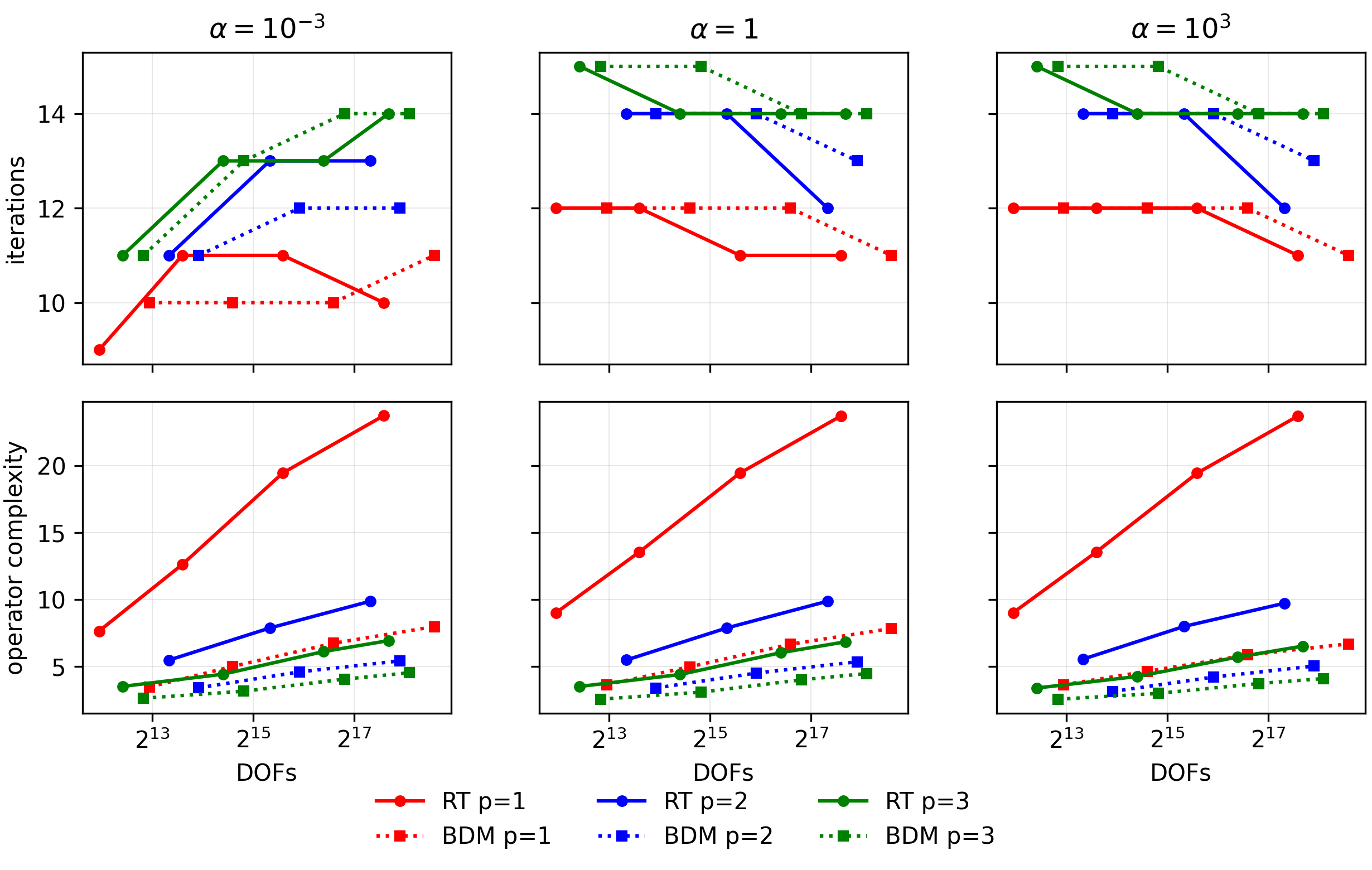}
    \caption{Two-dimensional \(H(\div;\mathcal D)\) grad--div results for RT
    and BDM discretizations.  The columns correspond to
    \(\alpha=10^{-3},1,10^3\).  The top row shows iteration counts and the
    bottom row shows operator complexity.}
    \label{fig:hdiv_results}
\end{figure}

In the tests below, the source term is chosen so that
\(\bm u_\star=-\grad(\cos(\pi x)\cos(\pi y))\) is the exact solution of the PDE, namely
\(\bm f=-\grad(\alpha\div\bm u_\star)+\beta\bm u_\star\).
The discretization uses triangular meshes and conforming Raviart--Thomas (RT)
and Brezzi--Douglas--Marini (BDM) spaces. With the Firedrake degree convention,
\(\mathrm{RT}_p\) contains vector polynomials through degree \(p-1\) together
with the usual Raviart--Thomas enrichment, whereas
\(\mathrm{BDM}_p=\mathbb P_p(T)^2\) contains all vector polynomials of degree
\(p\) on each triangle \(T\). 
We fix \(\beta=1\) in \eqref{eq:hdiv_weak_penalty_form} and vary
\(\alpha\in\{10^{-3},1,10^3\}\), so that \(\alpha\) controls the strength of
the grad--div term relative to the reaction term in \eqref{eq:hdiv_strong_graddiv_problem}.
In the tests we use the adaptive cutoff
\(
    \tau_{\rm cut}=0.75\max_j {\rm mult}_{\omega}(j).
\)

Numerical results are shown in \cref{fig:hdiv_results}, where we report
iteration counts and operator complexity as functions of the total number of DOFs. 
Overall, across RT and BDM spaces, all values of $\alpha$, and orders 1--3, we see fast and robust convergence of LS-AMG-DD as a standalone solver. 
Iteration counts are nearly identical for equal-order RT and BDM discretizations, and are largely scalable in both \(\alpha\) and mesh spacing, with mild growth in operator complexities as the mesh is refined, with the exception of lowest-order RT, which is relatively larger and grows much more quickly.
We observe some mild increase in iteration count with increasing \(p\), but simultaneously lower operator complexities, particularly for RT, where the $p=1$ operator complexity is significantly larger than $p\in\{2,3\}$.

\begin{figure}[t!]
    \centering
    \includegraphics[width=0.85\linewidth]{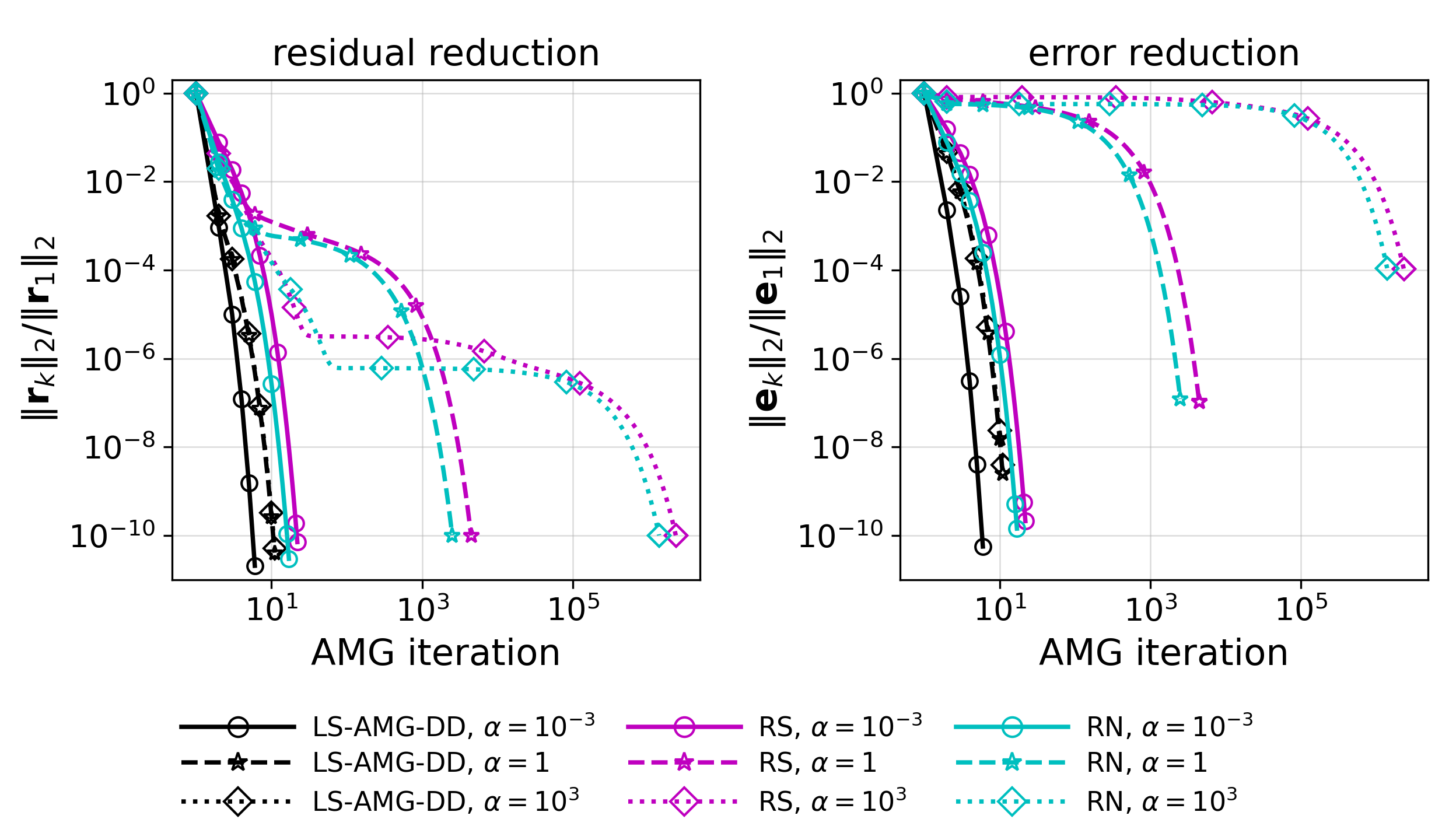}
    \caption{Comparison of LS-AMG-DD with two baseline AMG solvers for the grad--div problem using BDM $p=1$ elements. The number of DOFs is 416.}
    \label{fig:hdiv_pyamg_baseline}
\end{figure}

In \cref{fig:hdiv_pyamg_baseline} we show a comparison of LS-AMG-DD with two-level RS AMG, and RN AMG solvers for a fixed small mesh size and varying $\alpha \in \{ 10^{-3}, 1, 10^{3} \}$. Note that the number of DOFs in \cref{fig:hdiv_pyamg_baseline} is only 416, while the smallest BDM$_1$ results shown in \cref{fig:hdiv_results} use 6272 DOFs.
Evidently, the latter two solvers outright fail for $\alpha \in \{ 1,10^3 \}$, with the iterations increasing dramatically with $\alpha$, taking more than $10^6$ iterations for $\alpha = 10^3$, while those of LS-AMG-DD are $\leq 10$ and largely independent of $\alpha$. Moreover, despite obtaining comparable final relative residuals, the error achieved by LS-AMG-DD for $\alpha = 10^3$ is \emph{four orders of magnitude} smaller than that obtained by RS AMG or RN AMG. This is in contrast to $\alpha=10^{-3}$, where RS AMG and RN AMG converge comparably to LS-AMG-DD and also obtain comparable errors. The discrepancy in error for the ill-conditioned regime of large $\alpha$ emphasizes the importance of a robust preconditioner that can uniformly attenuate all error modes in obtaining an accurate solution to ill-conditioned problems.
\subsection{Scalar \texorpdfstring{\(H^2\)}{H2}: anisotropic hyperdiffusion}
\label{subsec:h2_hyperdiffusion_problem}

We next consider a fourth-order anisotropic biharmonic problem in \(H^2(\mathcal D)\), obtained by ``squaring'' an anisotropic diffusion operator. Let \(\mathcal D=(0,1)^2\), and define
\[
    H^2(\mathcal D)
    =
    \{v\in L^2(\mathcal D):
    \partial^\alpha v\in L^2(\mathcal D)
    \text{ for every }
    \alpha\in\mathbb N_0^2,\ |\alpha|\le 2\}.
\]
Let \(\bm b:\mathcal D\to\mathbb R^2\) be a unit direction field,
\(\|\bm b(\bm x)\|_2=1\), and define the diffusion tensor
\begin{equation}
\label{eq:h2_diffusion_tensor}
    D(\bm x)
    =
    \epsilon I
    +
    (1-\epsilon)\bm b(\bm x)\bm b(\bm x)^{\top},
    \qquad 0<\epsilon\le 1 .
\end{equation}
Then \(\epsilon=1\) gives the isotropic tensor, while \(\epsilon < 1\)
gives a tensor representing unit diffusion parallel to \(\bm b\) and diffusion of size \(\epsilon\)
in directions orthogonal to \(\bm b\).  

Define the second-order operator
\(q(w)=\operatorname{div}(D\operatorname{grad} w)\).  The strong version of the fourth-order problem we consider is then
\begin{align} \label{eq:h2_strong}
    \operatorname{div}
    \bigl(
    D\operatorname{grad} 
    \operatorname{div}(D\operatorname{grad} u) 
    \bigr)
    =
    q(q(u))=f
    \;\text{in }\mathcal D,
    \qquad
    u=0,\quad \partial_{\bm n}u=0
    \;\text{on }\partial\mathcal D,
\end{align}
where \(\partial_{\bm n}w=\operatorname{grad}w\cdot\bm n\).
Note that for $D = I$, this is the standard biharmonic equation,
\(
\Delta^2 u = f.
\)
We point the interested reader to \cite{krzysik2026bras} wherein we apply LS-AMG-DD to a Schur complement arising from a mixed discretization of a different fourth-order operator in the context of studying mass-matrix preconditioning.
We also write
\(\partial_{D,\bm n}w=(D\operatorname{grad}w)\cdot\bm n\) 
for the conormal derivative associated with \(q\).  Testing with \(v \in H^2({\cal D})\) and integrating by parts twice gives
\begin{align*}
    (q(u),q(v))_{\mathcal D}
    -
    \langle q(u),\partial_{D,\bm n}v\rangle_{\partial\mathcal D}
    +
    \langle \partial_{D,\bm n}q(u),v\rangle_{\partial\mathcal D}
    =
    (f,v)_{\mathcal D}.
\end{align*}
Here we use the unrestricted conforming space \(V_h\subset H^2(\mathcal D)\), and impose the clamped data weakly by a symmetric Nitsche formulation: find \(u_h\in V_h\) such that
\(
    a_h(u_h,v_h)=\ell_h(v_h)
    \; \forall v_h\in V_h,
\)
with \(\ell_h(v_h)=(f,v_h)_{\mathcal D}\) and
\begin{align*}
    a_h(u_h,v_h)
    :={}&
    (q(u_h),q(v_h))_{\mathcal D}
    +
    \langle \gamma_0 u_h,v_h\rangle_{\partial\mathcal D}
    +
    \langle \gamma_1\partial_{\bm n}u_h,\partial_{\bm n}v_h\rangle_{\partial\mathcal D}
    \\
    &-
    \langle q(u_h),\partial_{D,\bm n}v_h\rangle_{\partial\mathcal D}
    +
    \langle \partial_{D,\bm n}q(u_h),v_h\rangle_{\partial\mathcal D}
    \\
    &-
    \langle q(v_h),\partial_{D,\bm n}u_h\rangle_{\partial\mathcal D}
    +
    \langle \partial_{D,\bm n}q(v_h),u_h\rangle_{\partial\mathcal D}.
\end{align*}
The final two lines contain the consistency and adjoint-consistency terms, written in paired form to make the symmetry of \(a_h\) explicit.  On each boundary facet \(F\), with diameter \(h_F\), we take
\(
    \gamma_{0,F}
    =
    \tfrac{24p^4}{h_F^3} \|D\bm n\|_2^2,
\)
and
\(
    \gamma_{1,F}
    =
    \tfrac{24p^2}{h_F} \|D\bm n\|_2^2,
\)
where \(p\) is the polynomial degree.

In the numerical tests, the source is \(f(\bm x)=\exp(-\|\bm x-(0.5,0.5)\|^2/(2\sigma^2))\), with \(\sigma=0.045\).
Further, we consider the diffusion tensor \eqref{eq:h2_diffusion_tensor} with the following two direction fields
\begin{equation}
\label{eq:h2_bfield_cases}
    \bm b(\bm x)
    =
    \begin{cases}
    (\cos\theta,\sin\theta)^{\top},
        & \text{constant direction},
    \\
    \displaystyle
    \frac{(1,y_c'(x))^{\top}}{\sqrt{1+(y_c'(x))^2}},
    \qquad
    y_c(x)=x+a_S\sin(2\pi x),
        & \text{S-curve direction}.
    \end{cases}
\end{equation}
Specifically, we use \(\theta=\pi/6\) for the constant,
non-mesh-aligned direction and \(a_S=0.5\) for the S-curve direction.
The discretization uses triangular meshes and conforming \(C^1\) scalar finite-element spaces \(V_h\subset H^2(\mathcal D)\), given by  Argyris and Bell elements, corresponding to $p = 5$ \cite{kirby2019code}.  
\Cref{fig:h2_sol} shows numerical solutions to the problem; evidently the Gaussian blob is diffused much more strongly parallel to \(\bm b\) than orthogonal to it.

\begin{figure}[t!]
    \centering
    \includegraphics[width=0.425\linewidth]{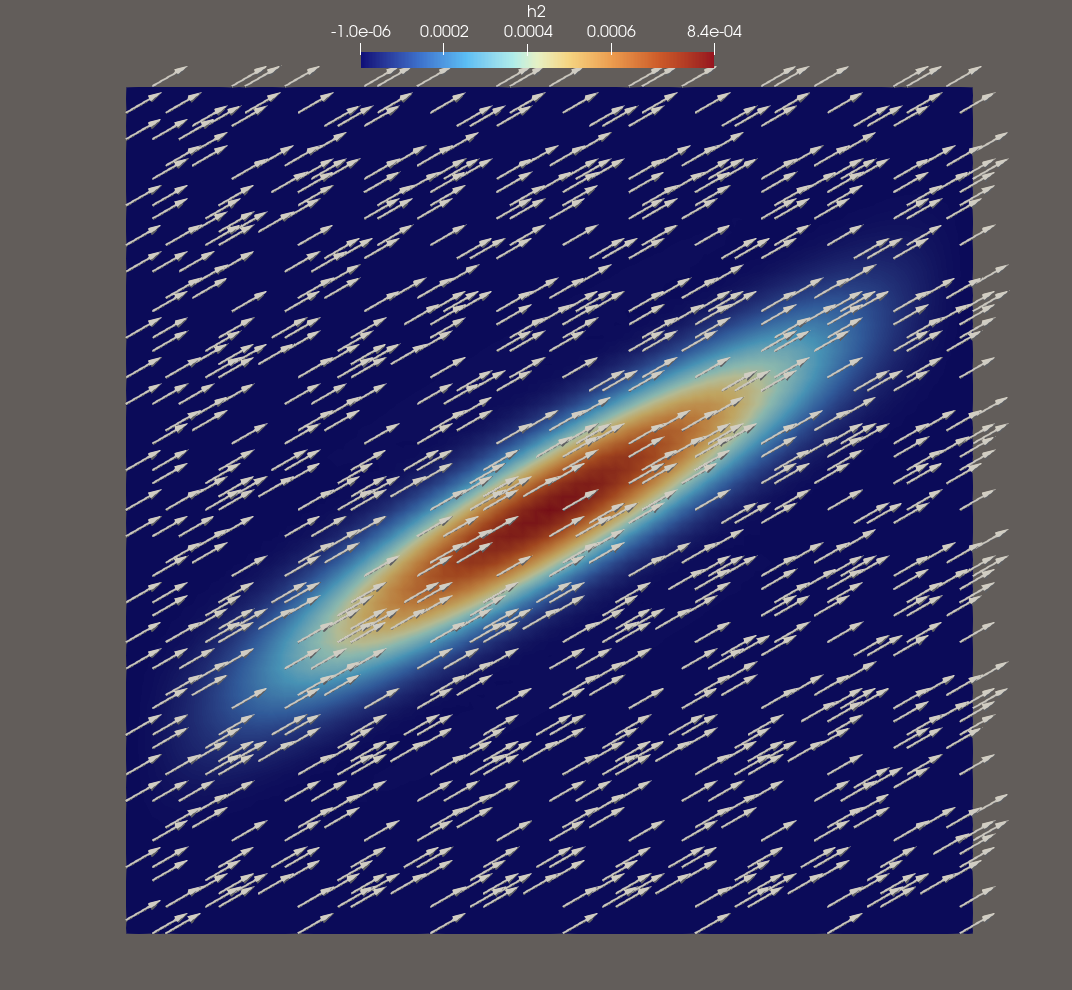}
    \quad
    \includegraphics[width=0.425\linewidth]{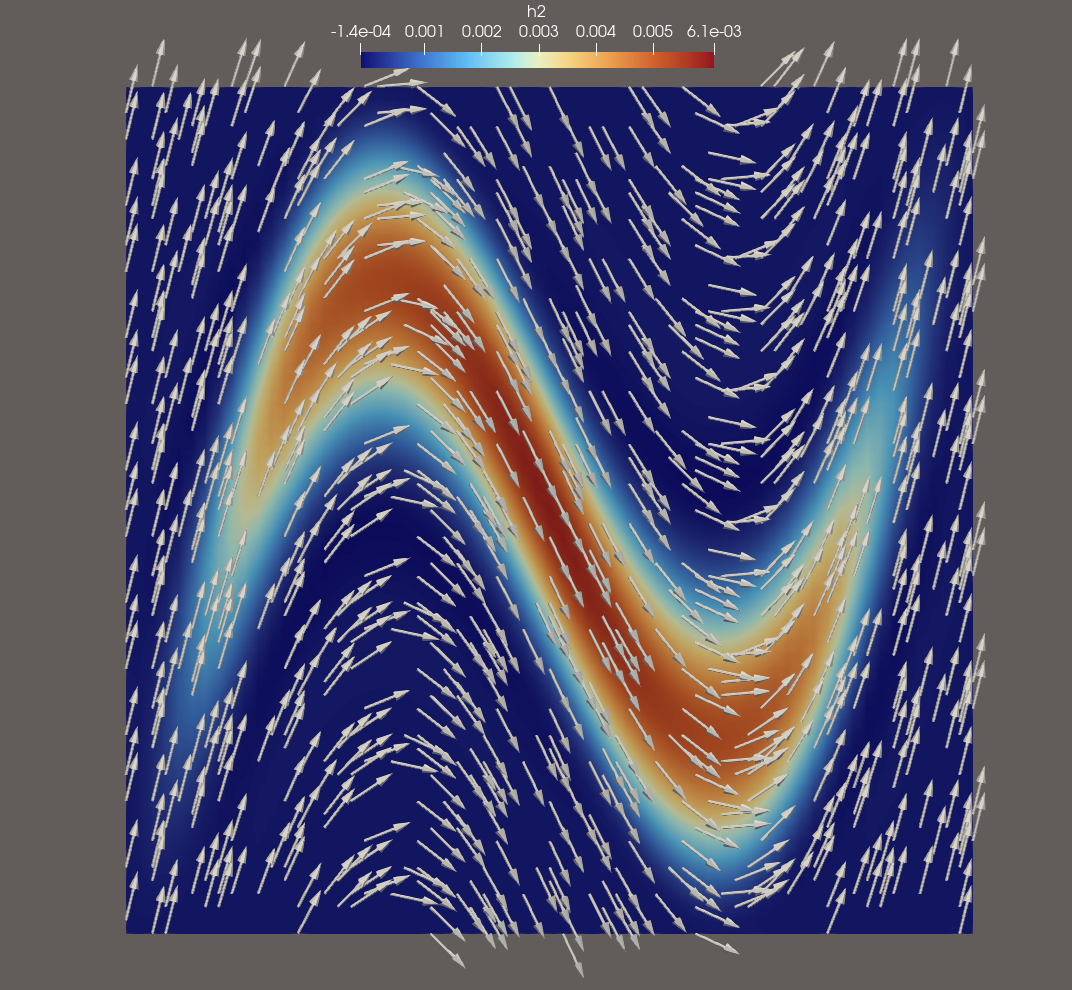}
    \caption{Solutions of the anisotropic biharmonic equation \eqref{eq:h2_strong} in $H^2$, with diffusion tensor \eqref{eq:h2_diffusion_tensor} using an anisotropy ratio \(\epsilon=10^{-3}\), and direction fields \(\bm b (\bm x)\) from \eqref{eq:h2_bfield_cases}.
    Left: constant direction field.  Right: S-curve direction field.
    Overlaid white arrows represent the direction of \(\bm b (\bm x)\).}
    \label{fig:h2_sol}
\end{figure}

\begin{figure}[t!]
    \centering
    \includegraphics[width=0.925\linewidth]{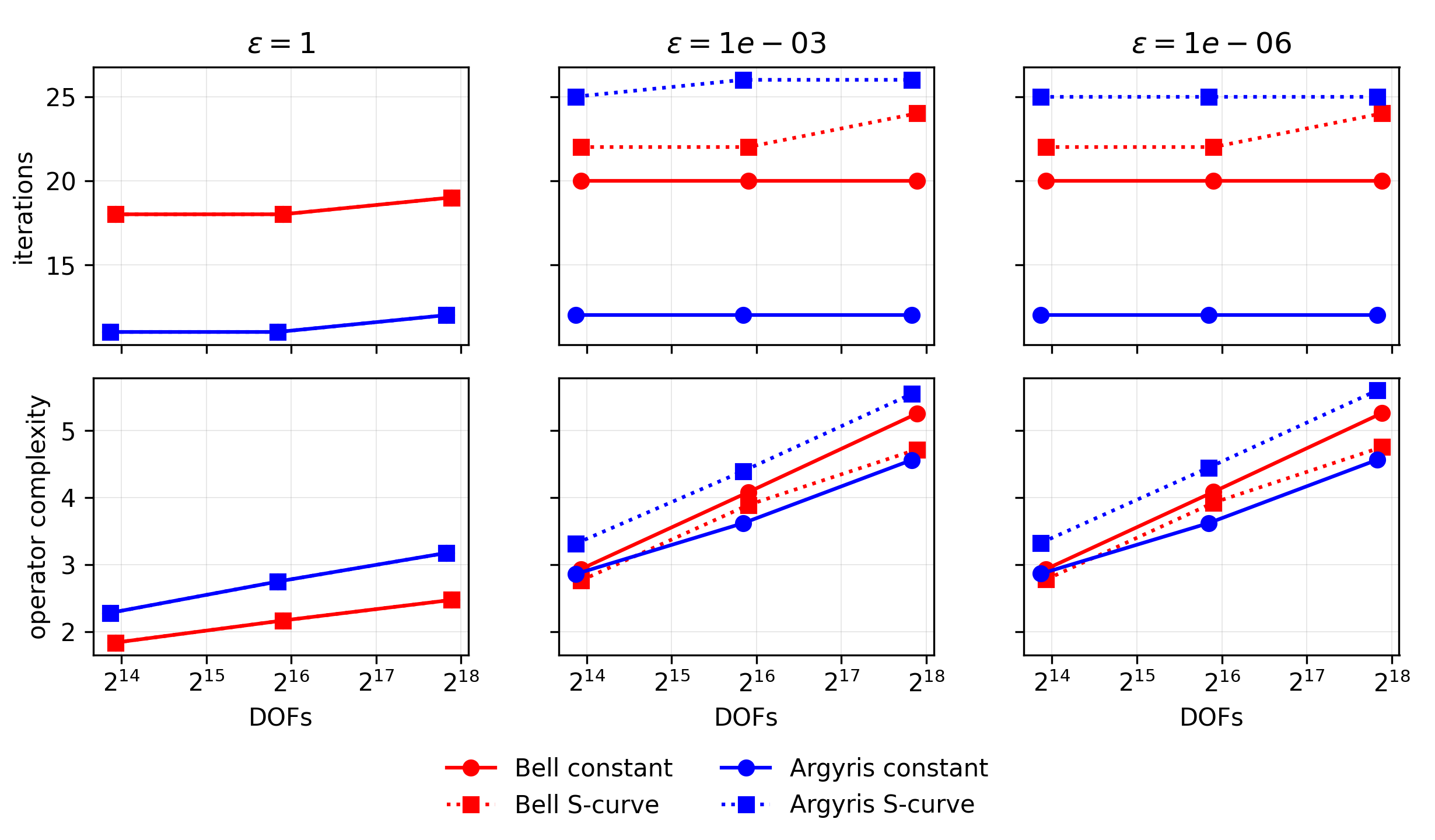}
    \caption{Results for the \(H^2\) problem \eqref{eq:h2_strong} with the two \(\bm b\) fields in
    \eqref{eq:h2_bfield_cases}.  The top row shows iteration counts and the
    bottom row shows operator complexity.}
    \label{fig:h2_results}
\end{figure}

\begin{figure}[b!]
    \centering
    \includegraphics[width=0.85\linewidth]{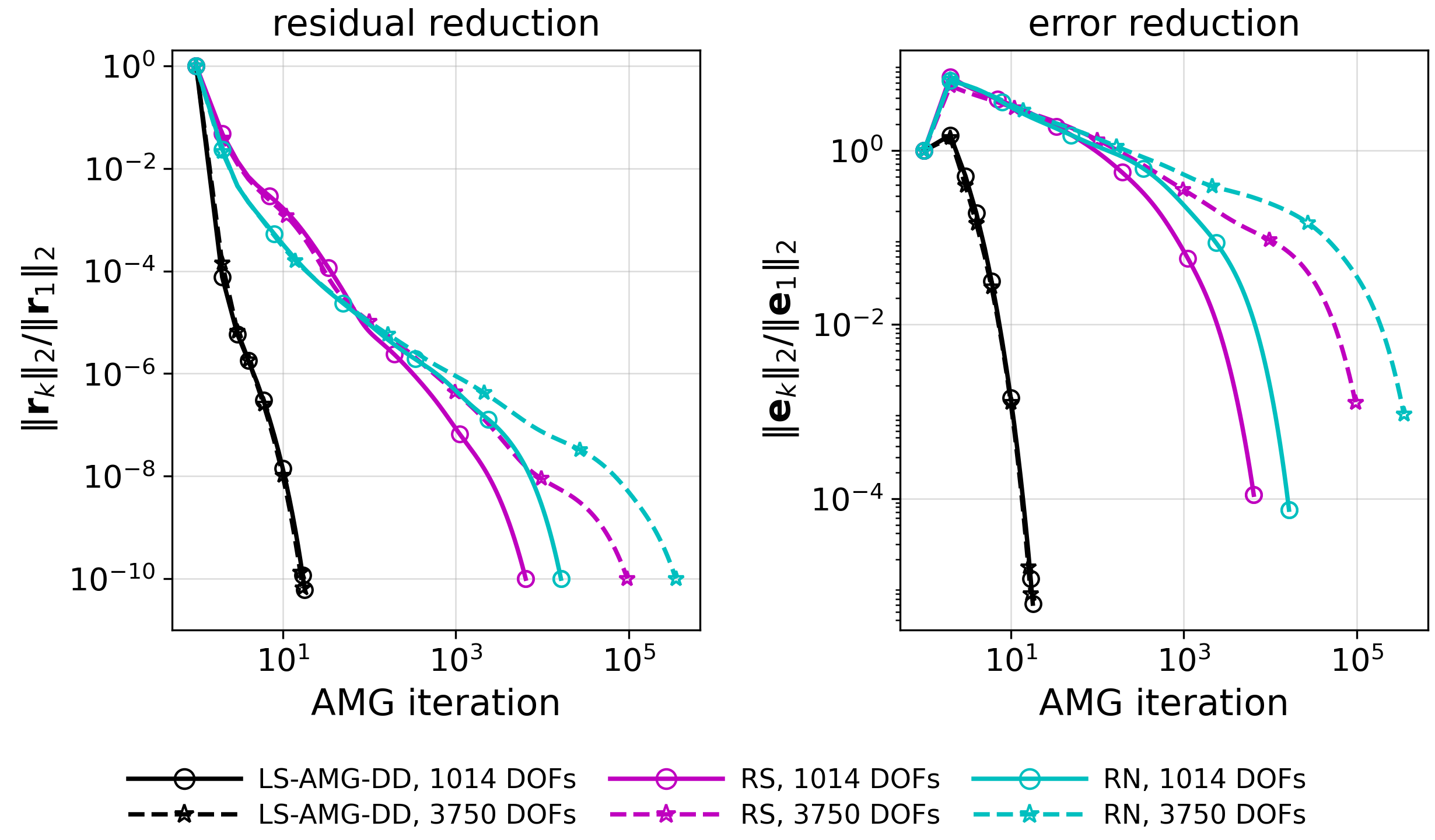}
    \caption{Comparison of LS-AMG-DD with the two baseline AMG solvers for the anisotropic biharmonic problem \eqref{eq:h2_strong} discretized with Bell elements, and with \(\epsilon=10^{-6}\) and the S-curve field \(\bm b (\bm x)\) from \eqref{eq:h2_bfield_cases}.}
\label{fig:h2_pyamg_baseline}
\end{figure}

The numerical experiments use anisotropy ratios in \eqref{eq:h2_diffusion_tensor} of \(\epsilon\in\{1,10^{-3},10^{-6}\}\).
The largest systems here have on the order of \(240\)k DOFs with \(10\)M nonzeros in \(A\).  
The coarsest systems in these runs have at most about 6,000 DOFs.
Solver results are shown in \cref{fig:h2_results}, which show fast, robust convergence across all regimes tested.
The iteration counts are largely constant under mesh refinement, while the
operator complexities grow noticeably, especially in the anisotropic
cases.
For the constant-direction field, iterations are also largely insensitive to anisotropy, seeing near identical results moving from an isotropic problem to highly anisotropic form. For the S-curve, there is some growth as the problem transitions from the isotropic case to
the moderately anisotropic case, \(\epsilon=1\to 10^{-3}\), but no further increase in iterations moving from \(\epsilon=10^{-3} \to 10^{-6}\). 
This behavior is tied to the adaptive threshold \(\tau_{\rm cut}=\max_j {\rm mult}_{\omega}(j)\) from \eqref{eq:tau-cut}, which attempts to ensure that coarse spaces are comparably robust regardless of anisotropy. Larger anisotropy simply requires more coarse basis functions, and therefore drives a higher operator complexity, while the isotropic problem maintains small operator complexities with slow growth under mesh refinement.

Two-level residual and error histories for small example problems are
shown in \cref{fig:h2_pyamg_baseline} to compare LS-AMG-DD against RS and RN AMG. Note that the largest DOF count of 3750 is still significantly smaller than the smallest Bell DOF count of 15606 shown in \cref{fig:h2_results}. First and foremost, we see that even on these very small problems using two-level methods, RS and RN AMG largely fail, taking $\sim 10^4$ and $10^5$ iterations to converge to $10^{-10}$ relative residual for the two problem sizes, respectively, demonstrating intractably poor convergence and serious degradation with mesh refinement. In contrast (and consistent with \cref{fig:h2_pyamg_baseline}), LS-AMG-DD converges rapidly, and identically across the two mesh refinements. Similar to what we saw for grad--div in \cref{fig:hdiv_pyamg_baseline}, despite all methods converging to the same final relative residual, LS-AMG-DD achieves {two orders of magnitude} smaller $\ell^2$ error on the finer problem than RS or RN AMG.

\subsection{Vector \texorpdfstring{\(H^1\)}{H1}: linear elasticity}
\label{subsec:h1_vector_linear_elasticity_problem}

The final model problem is two-dimensional small-strain isotropic linear
elasticity on the cantilever beam \(\mathcal D=(0,4)\times(0,1)\).  Let
\(\bm u:\mathcal D\to\mathbb R^2\) denote the displacement, and define
\begin{equation}
\label{eq:h1_elasticity_strain_stress}
    \varepsilon(\bm u)
    =
    \frac12\left(\grad\bm u+(\grad\bm u)^{\top}\right),
    \qquad
    \sigma(\bm u)
    =
    2\mu\varepsilon(\bm u)
    +
    \lambda(\div\bm u)I .
\end{equation}
Here \(\mu>0\) is the shear modulus and \(\lambda\ge 0\) is the first
Lam\'e parameter.  Let \(\Gamma_{\rm L}\), \(\Gamma_{\rm R}\),
\(\Gamma_{\rm B}\), and \(\Gamma_{\rm T}\) denote the left, right, bottom, and top sides of \(\partial\mathcal D\), respectively.  The strong form of the equation for the displacement \(\bm u\) is
\begin{equation}
\label{eq:h1_elasticity_strong_problem}
\begin{aligned}
    -\div\sigma(\bm u) &= \bm 0
        &&\text{in }\mathcal D, \\
    \bm u &= \bm 0
        &&\text{on }\Gamma_{\rm L}, \\
    \sigma(\bm u)\bm n &= \bm t_{\rm R}
        &&\text{on }\Gamma_{\rm R}, \\
    \sigma(\bm u)\bm n &= \bm 0
        &&\text{on }\Gamma_{\rm B}\cup\Gamma_{\rm T}.
\end{aligned}
\end{equation}

Testing the PDE with \(\bm v\in H^1(\mathcal D)^2\) and integrating by parts gives
\[
    (\sigma(\bm u),\grad\bm v)_{\mathcal D}
    -
    \langle \sigma(\bm u)\bm n,\bm v\rangle_{\partial\mathcal D}
    =
    0.
\]
Here we use the unrestricted conforming space \(V_h\subset H^1(\mathcal D)^2\), and impose the clamp
on \(\Gamma_{\rm L}\) weakly by a symmetric Nitsche formulation.  
Applying the boundary conditions, and noting that \(\sigma(\bm u) \, : \, \grad \bm v = \sigma( \bm u ) \, : \, \varepsilon( \bm v )\), the discrete problem is: find \(\bm u_h\in V_h\) such that
\(
    a_h(\bm u_h,\bm v_h)=\ell_h(\bm v_h)
    \;
    \forall \bm v_h\in V_h,
\)
with
\(
    \ell_h(\bm v_h)
    =
    \langle \bm t_{\rm R},\bm v_h\rangle_{\Gamma_{\rm R}},
\)
and
\begin{align}
\label{eq:h1_vector_elasticity_nitsche_form}
    a_h(\bm u_h,\bm v_h)
    :={}&
    2\mu(\varepsilon(\bm u_h),\varepsilon(\bm v_h))_{\mathcal D}
    +
    \lambda(\div\bm u_h,\div\bm v_h)_{\mathcal D}
    +
    \langle \gamma\bm u_h,\bm v_h\rangle_{\Gamma_{\rm L}}
    \nonumber\\
    &-
    \langle \sigma(\bm u_h)\bm n,\bm v_h\rangle_{\Gamma_{\rm L}}
    -
    \langle \sigma(\bm v_h)\bm n,\bm u_h\rangle_{\Gamma_{\rm L}} .
\end{align}
On each boundary facet \(F\subset\Gamma_{\rm L}\), with diameter \(h_F\), we take
\(
    \gamma_F
    =
    \tfrac{24p^2}{h_F}(\lambda+2\mu).
\)

\begin{figure}[b!]
    \centering
    \includegraphics[width=0.475\linewidth]{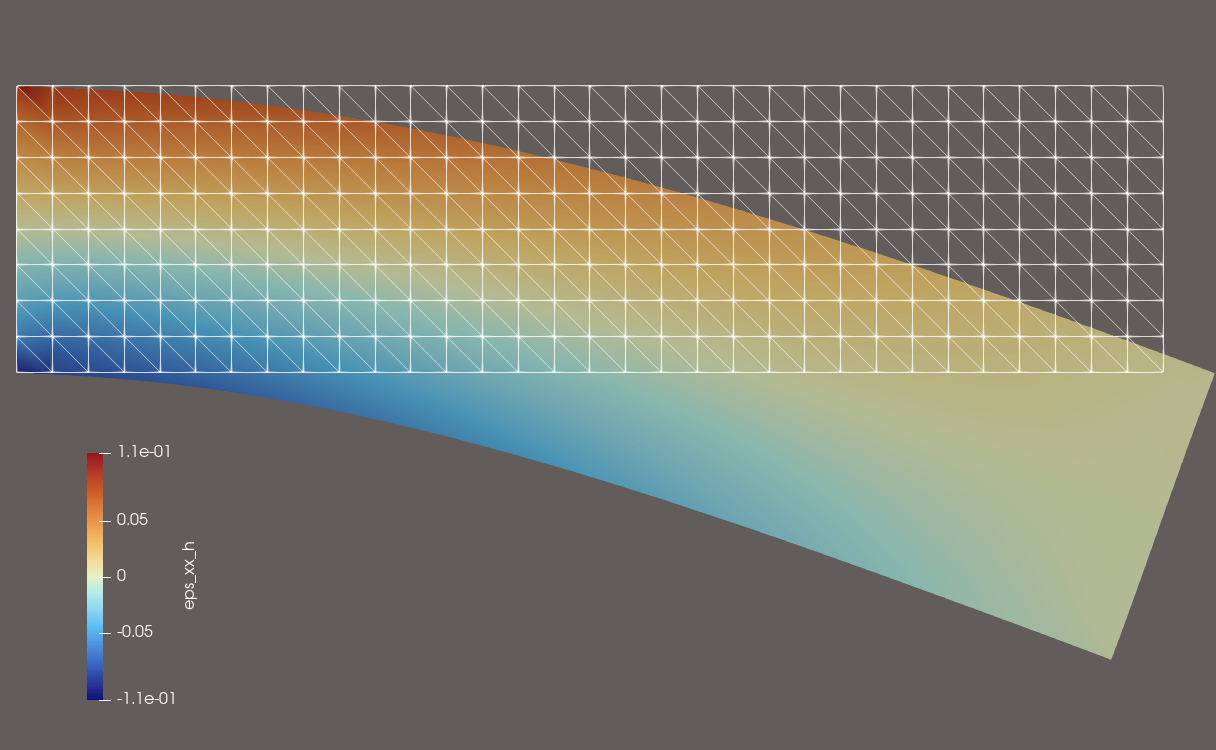}
    \;
    \includegraphics[width=0.475\linewidth]{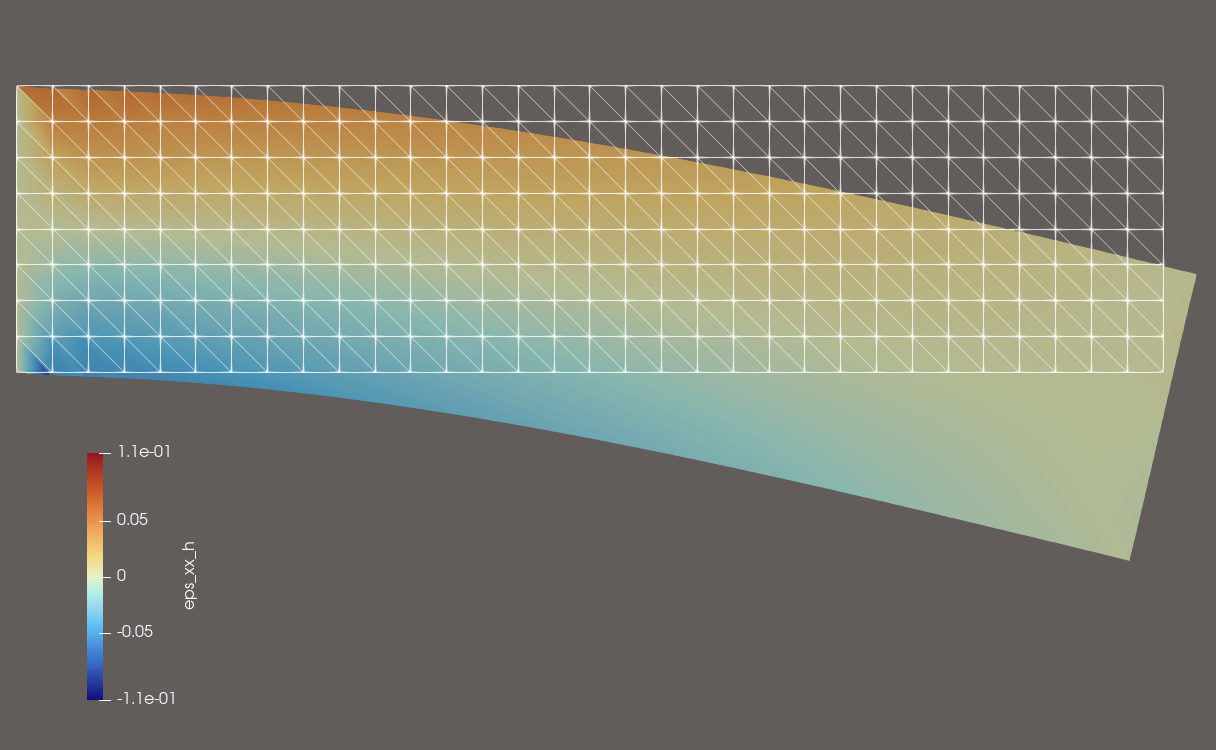}
    \caption{Displaced cantilever-beam solutions for the vector
    \(H^1\) linear elasticity problem \eqref{eq:h1_elasticity_strong_problem}, 
    with the color field showing the axial strain component
    \(\varepsilon_{xx}=\partial u_x/\partial x\).
    The finite-element space is \(V_h=[\mathrm{CG}_p(\mathcal T_h)]^2\) with $p=2$.
    Left: \(\lambda/\mu=1\).
    Right: \(\lambda/\mu=499\).
    }
    \label{fig:elasticity_sol}
\end{figure}

The discretization uses triangular meshes with
\(V_h=[\mathrm{CG}_p(\mathcal T_h)]^2\). Letting \(N\) denote
the number of cells through the beam height, the axial cell count is
\(4N\), matching the \(4:1\) beam aspect ratio.  
The material sweep fixes
\(\mu=1\) and varies
\begin{equation}
\label{eq:h1_elasticity_material_sweep}
    \frac{\lambda}{\mu}\in\{1,49,499\}.
\end{equation}
This gives one moderately compressible case and two increasingly
near-incompressible cases.  Under the usual plane-strain/three-dimensional
Lam\'e relation, these ratios correspond to Poisson ratios
\(\nu=0.25,0.49,0.499\), respectively.  The ratio \(\lambda/\mu\) controls the
strength of the volumetric term relative to the shear term, and the singular
limit \(\lambda/\mu\to\infty\) is the incompressible limit.
Finally, the traction boundary condition in \eqref{eq:h1_elasticity_strong_problem} is taken as \(\bm t_{\rm R} = (0,-T_0)^{\top}\) with \(T_0=10^{-2}\).
\Cref{fig:elasticity_sol} shows the displaced beam for
\(\lambda/\mu=1\) and \(\lambda/\mu=499\).

\begin{figure}[b!]
    \centering
    \includegraphics[width=0.95\linewidth]{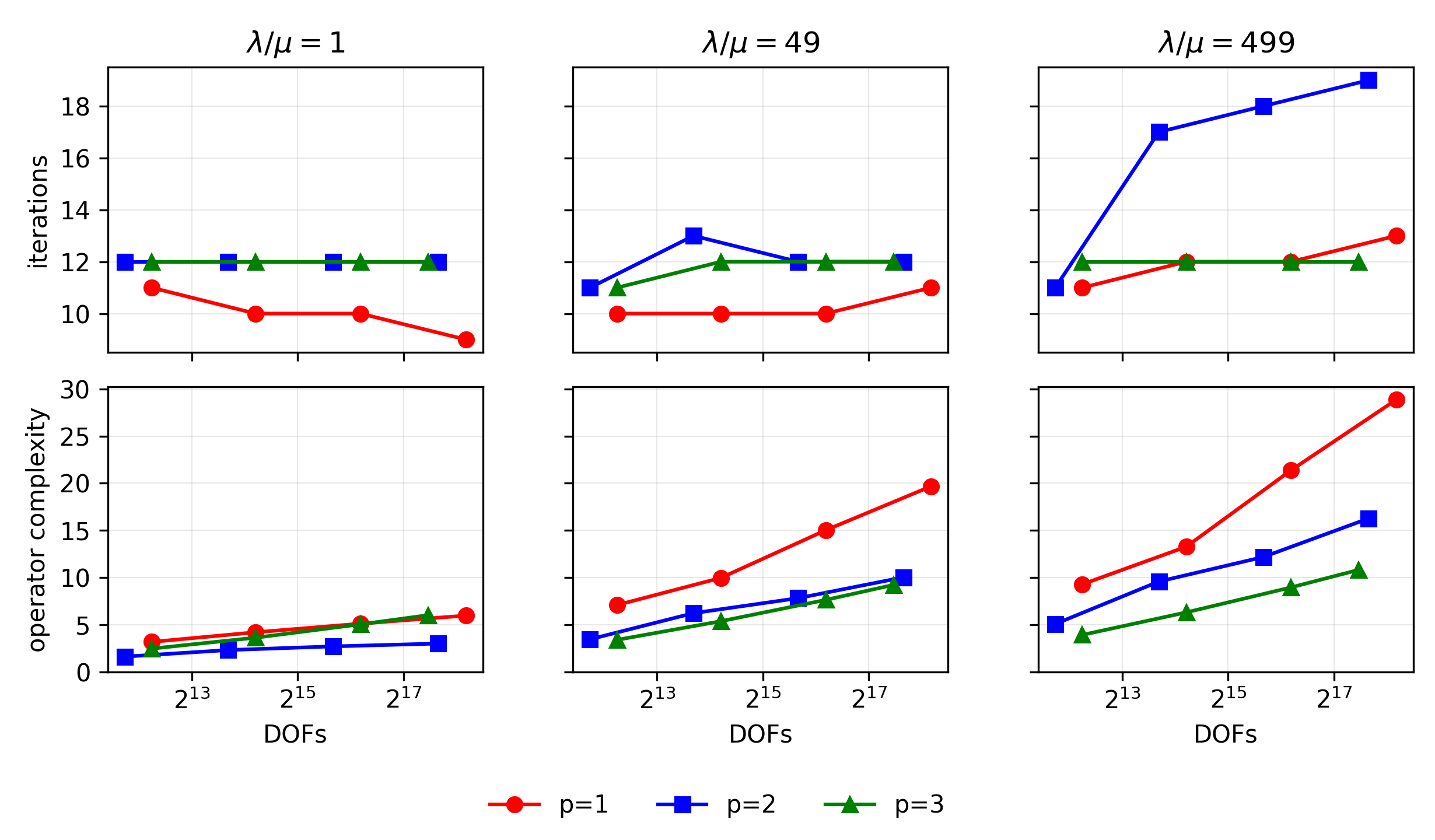}
    \caption{Results for linear elasticity \eqref{eq:h1_elasticity_strong_problem} discretized in \(V_h=[\mathrm{CG}_p(\mathcal T_h)]^2\), with
    \(\mu=1\) and \(\lambda/\mu\in\{1,49,499\}\).  The columns correspond to increasing values of \(\lambda/\mu\).  Top row: iteration count as a function of DOFs.  Bottom row: operator complexity as a function of DOFs.
    }
    \label{fig:h1_vector_results}
\end{figure}

Numerical results are shown in \cref{fig:h1_vector_results} for $p\in\{1,2,3\}$ as a function of mesh refinement.  We use the adaptive threshold
\(\tau_{\rm cut}=0.75 \max_j {\rm mult}_{\omega}(j)\). We find that the constant 0.75 yields more scalable iteration counts in mesh refinement than a value of 1.0. 
The iteration counts increase slowly as
\(\lambda/\mu\) grows, but remain flat under mesh refinement for
each fixed polynomial degree, except for $p=2$ and $\lambda/\mu=499$. In all cases, operator complexity increases with mesh size, especially in
near-incompressible cases and at low order. Interestingly, for $\lambda/\mu > 1$, the highest order we consider, $p=3$, yields the smallest iteration counts and operator complexities. 
This contrasts with usual AMG methods, where higher-order discretizations induce more complex near-nullspace modes and corresponding degradation of convergence. Here the adaptive threshold separates the compressible and
near-incompressible cases effectively, giving smaller coarse spaces for the
compressible case and more enriched coarse spaces when the volumetric term is
dominant.

\cref{fig:h1_vector_pyamg_baseline} presents an error and residual comparison with RS and RN AMG. Even for the easiest case of $\lambda/\mu=1$, RS and RN AMG take more than $10^4$ iterations to converge, with this number exceeding $10^6$ for $\lambda/\mu=499$. LS-AMG-DD is robust across Poisson ratio, converging in 10 iterations for all three cases. As before, LS-AMG-DD also provides significantly higher quality solutions, in this case achieving error more than five orders of magnitude smaller than the solutions of RS and RN AMG, despite converging to the same relative residual.

\begin{figure}[t!]
    \centering
    \includegraphics[width=0.85\linewidth]{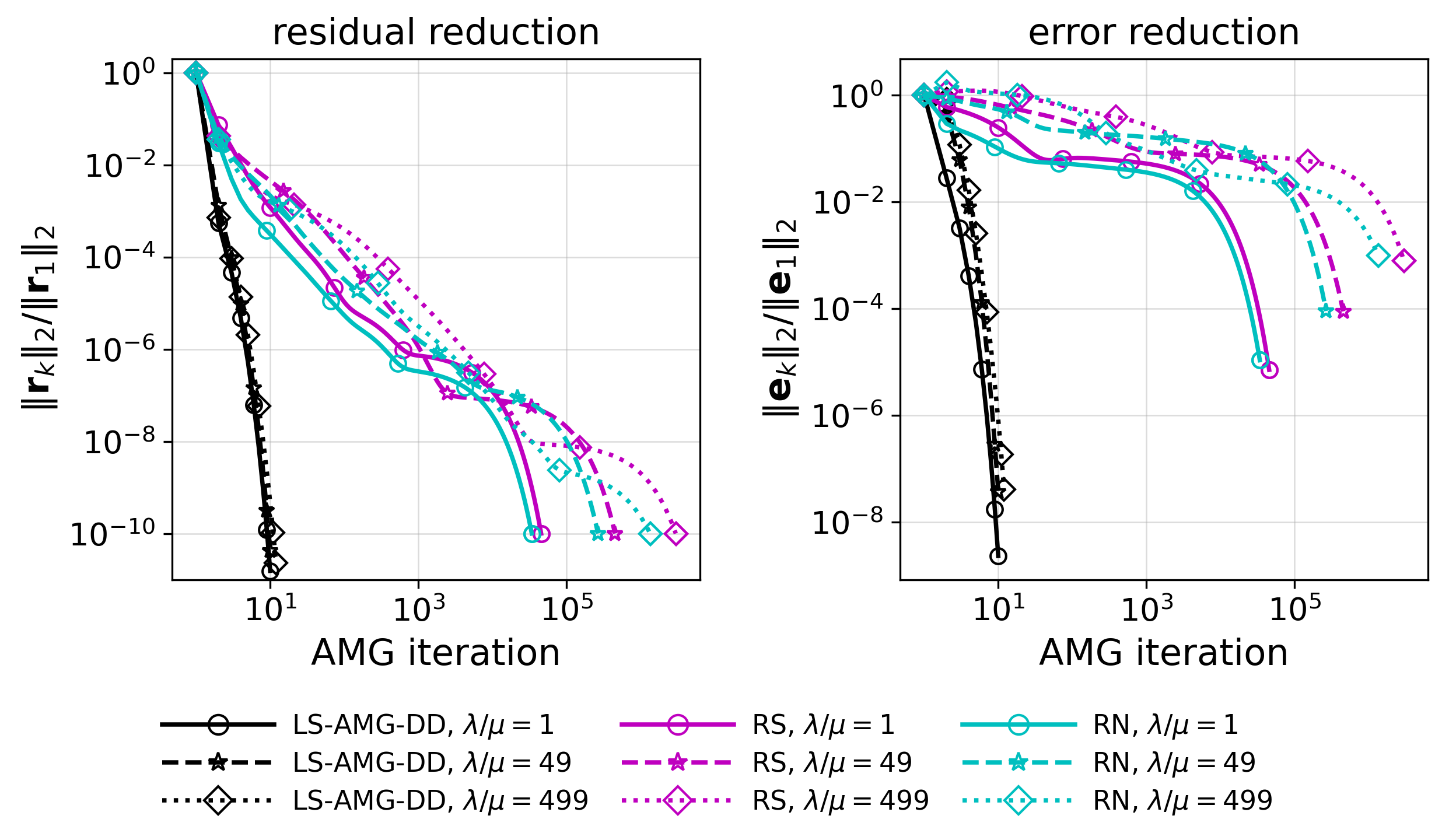}
    \caption{Comparison of LS-AMG-DD with the two baseline AMG solvers for the linear elasticity problem \eqref{eq:h1_elasticity_strong_problem} discretized in \(V_h=[\mathrm{CG}_p(\mathcal T_h)]^2\), $p=2$. The number of DOFs is 594.}
    \label{fig:h1_vector_pyamg_baseline}
\end{figure}

\section{Conclusions}
\label{sec:conclusions}

In this paper we showed that conforming finite-element discretizations can naturally be expressed as a sparse Gram matrix $A = G^\top G$, exactly the structure required by the multilevel algebraic solver LS-AMG-DD \cite{southworth2026lsamgdd}. On a prescribed DOF cover, we proved that the existence of a cover-local Gram form \(A=G^{\top}G\) is equivalent to the existence of an exact cover-local SPSD splitting of $A$. For conforming finite-element discretizations, such splittings arise directly from standard elementwise assembly.  Factoring local element matrices and stacking the resulting factors therefore gives an element-local Gram representation of the assembled operator.

We then couple this problem formulation with the LS-AMG-DD solver and a proposed local eigenvalue threshold as a robust black-box multilevel solver for conforming finite-element discretizations, beyond scalar \(H^1\) systems.  Across \(H(\div)\) grad--div, scalar \(H^2\) anisotropic hyperdiffusion, and vector \(H^1\) elasticity examples, the proposed solver framework provides fast and robust convergence (without Krylov acceleration) under mesh refinement, increasing polynomial degree, and difficult physical parameter regimes, where standard AMG baselines fail to converge in $10^5$ or more iterations. Moreover, the importance of a good preconditioner is apparent in error plots for the hardest/most ill-conditioned realizations of each problem. Although all methods are stopped at the same residual tolerance, LS-AMG-DD yields a final error that is 2--5 orders of magnitude smaller than the other methods because it uniformly attenuating error associated with all modes. Future work will focus on complexity control, because the same spectral enrichment that gives robustness can also produce large coarse spaces and significant fill-in on coarser grids.

\section*{Acknowledgments}
The authors used OpenAI's ChatGPT during the preparation of this manuscript, including for exploratory mathematical discussion, coding, and drafting of the exposition.
All mathematical claims, proofs, computations, citations, and final wording were reviewed, verified, and edited by the authors, who take full responsibility for the manuscript. We would like to thank Hussam Al Daas for useful conversations regarding spectral DD, as well as David Ham and Firedrake contributors for useful suggestions regarding broken spaces and unassembled matrices in Firedrake.

\bibliographystyle{siamplain}
\bibliography{cfem-solvers-refs}

\appendix

\section{Constructing the Gram factor in Firedrake}
\label{subsec:numerical_gram_construction}
The numerical experiments assemble finite-element forms in Firedrake
\cite{FiredrakeUserManual}, but Firedrake's public assembly interface does not expose element matrices. Instead, we use an equivalent broken-space construction to recover the same element-local information. Let \(V_h\) denote the conforming finite-element space and let
\(V_h^{\rm br}\) denote the broken companion space obtained from the same
local element. The conforming matrix \(A\) is obtained by assembling \(a_h(\cdot,\cdot)\)
on \(V_h\).  The same UFL form is then assembled on \(V_h^{\rm br}\), giving
a broken-space matrix \(A_{\rm br}\), with element-local DOFs and
block diagonal matrix \(A_{\rm br}\).
Boundary-facet terms in \(a_h\), when present, are included in the local blocks associated with boundary elements.
The element-local blocks of \(A_{\rm br}\) are factored by converting each block to a dense array and computing a right square-root by eigendecomposition.  We denote the resulting block-diagonal square-root  by
\(G_{\rm br}\), so that
\(
    A_{\rm br}=G_{\rm br}^{\top}G_{\rm br}.
\)
 Let
\(
    P_{\rm br}:V_h\to V_h^{\rm br}
\)
denote the map that takes a conforming coefficient vector to the
broken-space coefficient vector representing the same cellwise finite
element field. The rectangular Gram factor supplied to LS-AMG-DD is
\(
    G=G_{\rm br}P_{\rm br}
\)
where
\(
    G^{\top}G
    =
    P_{\rm br}^{\top}A_{\rm br}P_{\rm br}.
\)
In the Firedrake implementation, this broken-space approach  requires essential boundary conditions to be imposed weakly for consistency between \(V_h\) and \(V_h^{\rm br}\).

\end{document}